\documentclass[a4paper,reqno]{amsart}

\usepackage[utf8]{inputenc}
\usepackage{amssymb}
\usepackage{enumitem}
\usepackage{mathrsfs}
\usepackage{mathtools}
\usepackage{color}
\usepackage[colorlinks=true]{hyperref}
\hypersetup{urlcolor=green,citecolor=green}

\setlist[enumerate]{label=\emph{(\roman*)}}

\usepackage{environ}

\newtheorem{theorem}{Theorem}[section]
\newtheorem{corollary}[theorem]{Corollary}
\newtheorem{lemma}[theorem]{Lemma}
\newtheorem{proposition}[theorem]{Proposition}

\theoremstyle{definition}

\newtheorem{remark}[theorem]{Remark}
\numberwithin{equation}{section}
\newcommand{\R}{\mathbb{R}}

\def \d {{\rm{d}}}

\begin{document}

\parindent=0pt

	\title[Asymptotic behavior of solutions]
	{Asymptotic behavior of 2D wave--Klein-Gordon coupled system under null condition}
	
		\author[S.~Dong]{Shijie Dong}
	\address{Fudan University, School of Mathematical Sciences, 220 Handan Road, Shanghai, 200433, P.R. China.}
	\email{shijiedong1991@hotmail.com}
	
		\author[Y.~Ma]{Yue Ma}
	\address{Xi'an Jiaotong University, School of Mathematics and Statistics, 28 West Xianning Road, Xi'an Shaanxi 710049, P.R. China.}
	\email{yuemath@xjtu.edu.cn}

	\author[X.~Yuan]{Xu Yuan}
	\address{Department of Mathematics, The Chinese University of Hong Kong, Shatin, N.T., Hong Kong, P.R. China.}
	\email{xu.yuan@cuhk.edu.hk}
	\begin{abstract}
    We study the 2D coupled wave--Klein-Gordon systems with semilinear null nonlinearities $Q_0$ and $Q_{\alpha\beta}$. The main result states that the solution to the 2D coupled systems exists globally provided that the initial data are small in some weighted Sobolev space, which do not necessarily have compact support, and we also show the optimal time decay of the solution.
    
    The major difficulties lie in the slow decay nature of the wave and the Klein-Gordon components in two space dimensions, in addition, extra difficulties arise due to the presence of the null form $Q_0$ which is not of divergence form and is not compatible with the Klein-Gordon equations. To overcome the difficulties, a new observation for the structure of the null form $Q_{0}$ is required.
	\end{abstract}
	\maketitle
\section{Introduction}

\subsection{The model problem}
In this article, we consider the 2D coupled wave--Klein-Gordon system under null condition,
\begin{equation}\label{equ:waveKG}
\begin{aligned}
-\Box w&=C_{1}Q_{0}(w,v)+C_{1}^{\alpha \beta}Q_{\alpha\beta}(w,v),\quad (t,x)\in [0,\infty)\times \R^{2},\\
-\Box v +v&=C_{2}Q_{0}(w,v)+C_{2}^{\alpha \beta}Q_{\alpha\beta}(w,v),\quad (t,x)\in [0,\infty)\times \R^{2},
\end{aligned}
\end{equation}
where $C_{1},C_{2},C_{1}^{\alpha\beta},C_{2}^{\alpha\beta}\in \R$ and $Q_{0}$, $Q_{\alpha\beta}$ are the standard null forms:
\begin{equation*}
\begin{aligned}
Q_{0}(w,v)=\partial_{\alpha}w\partial^{\alpha}v\quad \mbox{and}\quad Q_{\alpha\beta}(w,v)=\partial_{\alpha}w \partial_{\beta}v-\partial_{\alpha}v\partial_{\beta}w.
\end{aligned}
\end{equation*}
The initial data are prescribed at $t=0$
\begin{equation}\label{eq:ID}
\big( w, \partial_t w, v, \partial_t v \big)_{{|t=0}}
=(w_{0},w_{1},v_{0},v_{1}).
\end{equation}
We use $\Box = \partial_\alpha \partial^\alpha = -\partial_t^{2}+ \partial_{1}^{2}+\partial_{2}^{2}$ to denote the d'Alembert operator. We use Latin letters $i, j, \cdots \in \{1, 2 \}$ to represent space indices, and use Greek letters $\alpha, \beta, \cdots \in \{0, 1, 2 \}$ to represent spacetime indices, and the Einstein summation convention for repeated upper and lower indices is adopted. 

\smallskip
For various kinds of wave-type equations and systems under a suitable restriction of dimension and nonlinearity, the local well-posedness of the Cauchy problem is well-known. Now, one question of basic importance is that whether the Cauchy problem admits a global-in-time solution provided the initial data are sufficiently smooth and small.
In the 3D case, this question was answered in the affirmative in the seminal works of Christodoulou \cite{Christodoulou} and Klainerman \cite{Klainerman86} on nonlinear wave equations with null nonlinearities, and of Klainerman \cite{Klainerman85} and Shatah \cite{Shatah} on nonlinear Klein-Gordon equations with general quadratic nonlinearities. Shortly after these fundamental results, the study on 3D coupled wave--Klein-Gordon systems started in~\cite{Bachelot, Georgiev}. 
On the other hand, in the 2D case, such study is somewhat harder due to the slower decay nature of the wave and the Klein-Gordon components compared to the 3D case (see for instance~\cite{Alinhac01b, OTT} and references therein).
In the present article, we aim to demonstrate the global existence for the 2D coupled wave--Klein-Gordon systems in some important cases, i.e., the semilinear null nonlinearities.

\subsection{Main result}
The goal of the present paper is twofold. The first one is to generalize the classical result \cite{Georgiev} of Georgiev on coupled wave--Klein-Gordon systems from 3D to 2D with more types of quadratic nonlinearities and non-compactly supported initial data, which is more challenging. The other one is that we expect to provide more tools and ideas which can be used to tackle other related physical models in 2D. More precisely, the main result of this article is the following:

\begin{theorem}\label{thm:main}
	Let $N\ge 14$ be an integer. There exists $0<\varepsilon_{0}\ll 1$ such that for all initial data $(w_{0},w_{1},v_{0},v_{1})$ satisfying the smallness condition
	\begin{equation}\label{est:small}
	\begin{aligned}
	&\sum_{k\le N+1}\left(\left\|\langle |x|\rangle^{k}\nabla^{k}w_{0}\right\|_{L^{1}+L^{2}}+\left\|\langle |x|\rangle^{k+1} \log(2+|x|) \nabla^{k}v_{0}\right\|_{L^{2}}\right)\\
	&+\sum_{k\le N}\left(\left\|\langle |x|\rangle^{k+1}\nabla^{k}w_{1}\right\|_{L^{1}+L^{2}}+\left\|\langle |x|\rangle^{k+2} \log(2+|x|) \nabla^{k}v_{1}\right\|_{L^{2}}\right)\le \varepsilon<\varepsilon_0,
	\end{aligned}
	\end{equation}
	the Cauchy problem~\eqref{equ:waveKG}--\eqref{eq:ID} admits a global-in-time solution $(w,v)$, which enjoys the following pointwise decay estimates
	\begin{equation}\label{est:thm}
	|w(t,x)|\lesssim \varepsilon\langle t\rangle^{-\frac{1}{2}},\ \
	|v(t,x)|\lesssim \varepsilon\langle t\rangle^{-1},\ \
	\left|\partial w(t,x)\right|\lesssim \varepsilon\langle t-|x|\rangle^{-\frac{3}{4}}\langle t\rangle^{-\frac{1}{2}}.
	\end{equation}
\end{theorem}

\begin{remark}
In Theorem~\ref{thm:main}, we can treat non-compactly supported initial data.
Besides, the pointwise asymptotic behavior \eqref{est:thm} of the solution $(w, v)$ to \eqref{equ:waveKG} is optimal in time, due to the reason that even for the 2D linear wave and Klein-Gordon a better decay in time is not expected.
\end{remark}

\begin{remark}
Compared with the previous work \cite{Dong2005}, here we can treat additionally the nonlinear term $Q_0(w, v)$, which is much harder to handle compared with $Q_{\alpha\beta}(w, v)$ due to two reasons. First, the term $Q_0(w, v)$ is not compatible with the Klein-Gordon equations. This means that we need to use the scaling vector field, which is not compatible with the Klein-Gordon equations, to obtain an extra $\langle t\rangle^{-1}$ decay from the null structure of $Q_0(w,v)$.
Second, the term $Q_0(w, v)$ is not of divergence form, which means we cannot express $Q_0(w, v)$ in the form of $\partial (\cdot)$, which prevents us from using some techniques of~\cite{Dong2005}.  
\end{remark}

\begin{remark}
In the remarkable work \cite{Stingo} of Stingo, a 2D quasilinear coupled wave and Klein-Gordon system with null nonlinearities has been shown to admit small global solution. In this work, the initial data do not need to be compactly supported, and the author used the microlocal analysis method to reduce the weights requirement on the initial data to be very low. See also the recent work Ifrim-Stingo~\cite{Ifrim-Stingo} for systems of 2D quasilinear wave--Klein-Gordon system.
\end{remark}

For the 3D coupled wave--Klein-Gordon systems, after the works~\cite{Bachelot, Georgiev}, fruitful progress has been made in~\cite{Ionescu-P-WKG,Kataya,PLF-YM-arXiv1}. 
In the 2D case, we refer to~\cite{Alinhac01b, Alinhac2, Cai,DoLeLe21,FangD, Yin, Li21, OTT} for the works on wave and Klein-Gordon equation, and to~\cite{Dong2005, Duan-Ma, Ma2017b, Stingo} for the study of coupled wave--Klein-Gordon systems.

\smallskip
The study of the coupled wave--Klein-Gordon systems has large physical significance and is mathematically challenging.
In the physical aspect, several fundamental physical models are governed by coupled wave--Klein-Gordon systems, and such models include the Dirac--Klein-Gordon model, the Klein-Gordon--Zakharov model, the Maxwell--Klein-Gordon model, the Einstein--Klein-Gordon model, and so on. To show these models are stable, one is required to study the relevant coupled wave--Klein-Gordon systems. In the mathematical aspect, the most well-known obstacle to study such coupled equations is that the scaling vector field $S = t\partial_t + x^i \partial_i$ is not compatible with the Klein-Gordon operator. Besides, the slow pointwise decay of 2D wave and Klein-Gordon components, compared with the 3D scenario, causes serious problems in the study. Thus new ideas are demanded in such study. Lastly, we also would like to draw one's attention to some of the relevant works in \cite{Bachelot,Dong-Ma,Dong-Wyatt,Ionescu-P-EKG,Klainerman-WY} for the above-mentioned models.

\subsection{Comparison between $Q_{0}$ and $Q_{\alpha\beta}$}
The difficulties in studying coupled wave--Klein-Gordon systems include the non-commutativity of the scaling vector field and the Klein-Gordon operator, the different decay properties of the wave and the Klein-Gordon components etc.,
and we lead one to \cite[Introduction]{Dong2005} for the discussion. Here we only discuss in detail about the extra difficulties caused by the term $Q_0(w, v)$ compared with the term $Q_{\alpha\beta}(w, v)$.

\smallskip
First, the term $Q_0(w, v)$ is not compatible with the Klein-Gordon equations. We recall the estimates (see for instance \cite[Section 3]{Sogge})
$$
|Q_{\alpha\beta} (w, v)|
\lesssim
\langle t\rangle^{-1} \big| \Gamma w \big| \, \big|\Gamma v \big|,
$$
in which $\Gamma \in \{\partial_\alpha, H_i, \Omega \}$ (see the definition in Section \ref{sec:pre}) which are compatible with the Klein-Gordon equations. On the other hand, the estimates on the term $Q_0(w, v)$ demands the scaling vector field $S$, which is not compatible with the Klein-Gordon equations (and that is the reason why we call $Q_0$ non-compatible), i.e.,
$$
|Q_0 (w, v)|
\lesssim
\langle t\rangle^{-1} \big( \big| \Gamma w \big| + \big| S w \big| \big) \big|\Gamma v \big|.
$$
Thus in this sense, it is harder to bound the term $Q_0(w, v)$ by using the vector field method.

\smallskip
Second, we recall that if the nonlinearities in the wave equation take divergence form, then in most cases we can derive better bounds on the wave component. For instance, we consider the following wave equation (we ignore the initial data for simplicity)
$$
-\Box u = \partial_{\alpha} f
$$
where $\alpha=0,1,2$ and $f$ is a sufficiently regular function.
For the solution $u$, we can decompose it into two parts $u = u_{1} + \partial_\alpha u_{2}$, in which $u_{1}, u_{2}$ are solutions to the following wave equations
$$
-\Box u_{1} = 0,
\qquad
-\Box u_{2} = f.
$$
In this case, we can get very good control of $\partial u$ by the relation
$$
\partial u = \partial u_{1} + \partial \partial_\alpha u_{2},
$$
as $u_{1}$ satisfies the linear homogeneous wave equation (see the estimates in \eqref{est:Linwave}), and $\partial \partial_\alpha u_{2}$ is part of the Hessian of $u_{2}$ which roughly speaking enjoys an extra $\langle t-|x|\rangle^{-1}$ decay compared with $\partial u_{2}$ (see the estimates in Lemma  \ref{lem:Hessian}).
This decomposition can be used to treat the term $Q_{\alpha\beta}(w, v)$, as it is of divergence form, i.e.,
$$
Q_{\alpha\beta}(w, v) = \partial_\beta (v \partial_\alpha w) - \partial_\alpha (v \partial_\beta w).
$$
However, since the term $Q_0(w, v)$ is not of divergence form, the above argument cannot be directly applied, thus it requires a new decomposition to handle the term $Q_0(w, v)$.

\subsection{New ingredients}
To prove Theorem \ref{thm:main}, one crucial step is to close the highest-order energy estimates for the solution. One natural choice is to apply the ghost weight energy estimates of Alinhac \cite{Alinhac01b}. We note that it is necessary to first show (see \eqref{est:GhostGu})
$$
|\partial w| \lesssim \langle t \rangle^{-1/2} \langle t-|x|\rangle^{-1/2-\delta}
$$
for some $\delta>0$, so that the ghost weight energy estimates can be applied. Unlike the case of $Q_{\alpha\beta}(w, v)$ in the wave equation which is of divergence form (note the nonlinearities in the wave equation of \cite{Dong2005, Stingo} are also of divergence form), and from which roughly speaking we can gain one more derivative and hence an extra $\langle t-|x|\rangle^{-1}$ decay on the wave component, we need to find a new strategy to treat the case of $Q_0(w, v)$ which is not of divergence form. 

The key to treat the nonlinear term $Q_0(w, v)$ is one simple but useful observation
$$
Q_0(w, v)
= \partial_\alpha w \partial^\alpha v
= \partial^\alpha (v \partial_\alpha w ) - v \partial_\alpha \partial^\alpha w,
$$
and we note that the first term in the right hand side is of divergence form. Then we note that the term $\partial_\alpha \partial^\alpha w = \Box w$ contains only quadratic terms which makes $v \partial_\alpha \partial^\alpha w$ into cubic terms. Furthermore, this way of decomposing the nonlinear term $Q_0(w, v)$ can be carried on so that we can express $Q_0(w, v)$ as the summation of quartic terms and terms of divergence form; see Lemma \ref{le:equQ0Qab} for more details.

Compared with the proof in \cite{Dong2005}, we do not need to rely on an iteration procedure, and the key is to apply the Sobolev inequality \eqref{est:SobolevGlobal} proved in \cite{Geor} by Georgiev. This makes the presentation of the proof in the present paper neater.

\subsection{Organisation}
The article is organized as follows.
In Section \ref{sec:pre}, we introduce the notation and some fundamental estimates and tools to be used in Section~\ref{sec:proof}: the estimates on commutators and null forms, global Sobolev inequality, and $L^{\infty}$ estimates for wave equations and Klein-Gordon equations. 
In Section \ref{sec:proof}, we provide the proof for Theorem \ref{thm:main} by Klainerman's vector field method enhanced with Alinhac's ghost weight method.

\subsection*{Acknowledgement}
The author S.D. was partially supported by the China Postdoctoral Science Foundation, with grant number 2021M690702.

\section{Preliminaries}\label{sec:pre}

\subsection{Notation}

We work in the (1+2) dimensional spacetime $\R^{1+2}$ with Minkowski metric $\eta= (-1,1,1)$ which is used to raise or lower indices.
For a point $(x_{0},x_{1},x_{2})=(t,x_{1},x_{2})\in \R^{1+2}$, we denote its spacial radius by $r=\sqrt{x_{1}^{2}+x_{2}^{2}}$. The space indices are denoted by Latin letters $i,j\in \{1,2\}$. The spacetime indices are denoted by Greek letters $\alpha,\beta\in \{0,1,2\}$. 

To state global Sobolev inequalities, we first introduce the following four groups of vector fields:
\begin{enumerate}
	\item [(i)] Translations: $\partial_{\alpha}=\partial_{x_{\alpha}}$, for $\alpha=0,1,2$.
	
	\item [(ii)] Spatial rotations: $\Omega =x_{1}\partial_{2}-x_{2}\partial_{1}$.
	
	\item [(iii)] Scaling vector field: $S=t\partial_{t}+x^{i}\partial_{i}$.
	
	\item [(iv)] Hyperbolic rotations: $H_{i}=t\partial_{i}+x_{i}\partial_{t}$, for $i=1,2$.
\end{enumerate}

Excluding the scaling vector field $S$, we consider a general vector field set 
\begin{equation*}
V=\left\{\Omega;\partial_{\alpha},\alpha=0,1,2;H_{i},i=1,2\right\}.
\end{equation*}
For future notational convenience, we order these vector fields that belong to $V$ in some arbitrary manner, and we label them as $\Gamma_{1},\Gamma_{2},\dots,\Gamma_{6}$. Moreover, for any multi-index $I=(I_{1},I_{2},\dots,I_{6})\in \mathbb{N}^{6}$, we denote
\begin{equation*}
\Gamma^{I}=\prod_{k=1}^{6}\Gamma_{k}^{I_{k}},\quad \mbox{where}\ \Gamma=\left(\Gamma_{1},\Gamma_{2},\dots,\Gamma_{6}\right).
\end{equation*}

In addition, we also introduce the good derivatives
\begin{equation*}
G_{i}=\frac{1}{r}\left(x_{i}\partial_{t}+r\partial_{i}\right)\quad \mbox{for}\ i=1,2.
\end{equation*}

The Fourier transform is defined as 
\begin{equation*}
\hat{h}(\xi)=\frac{1}{2\pi}\int_{\R^{2}}h(x)e^{-ix\cdot\xi}\d x,\quad \mbox{for}\ h\in L^{2}_{x}.
\end{equation*}

For $(x,\rho)\in \R^{2}\times \R_{+}$, we denote by $B(x,\rho)$ (respectively, $\partial B(x,\rho)$) the ball (respectively, the sphere) of $\R^{2}$ of center $x$ and of radius $\rho$.

Let $\left\{\psi_{j}\right\}_{j=0}^{\infty}$ be a Littlewood-Paley partition of unity, i.e.
\begin{equation*}
1=\sum_{j=0}^{\infty}\psi_{j}(s),\ s\ge 0,\ \psi_{j}\in C_{0}^{\infty}\left(\R^{2}\right),\quad \psi_{j}\ge 0\quad \mbox{for all}\ j\ge 0,
\end{equation*}
as well as 
\begin{equation*}
\mbox{supp}\psi_{0}\cap [0,\infty)=[0,2],\quad 
\mbox{supp}\psi_{j}\subset \left[2^{j-1},2^{j+1}\right]\quad \mbox{for all}\ j\ge 1.
\end{equation*}

We will use $Q$ to denote a general null form in $\left\{Q_{0};\ Q_{\alpha\beta}, 0\le \alpha\ne \beta\le 2\right\}$.

For simplicity of notation, we denote the initial data $(w_{0},w_{1},v_{0},v_{1})$ by 
\begin{equation*}
\vec{w}_{0}=(w_{0},w_{1}),\
\vec{v}_{0}=(v_{0},v_{1}),\
\left(\vec{w}_{0},\vec{v}_{0}\right)=(w_{0},w_{1},v_{0},v_{1}).
\end{equation*}

 We write $A \lesssim B$ to indicate $A \leq C B$ with $C$ a universal constant, and we use the notation $\langle s \rangle = \sqrt{1+|s|^2}$ for $s\in \R$.

\subsection{Estimates on commutators and null forms}
In this subsection, we state some preliminary estimates related to commutators and null forms $Q$.
We first recall the well-known relations
\begin{equation*}
\left[\Box,\Gamma_{k}\right]=\left[ (\Box-1),\Gamma_{k}\right]=0,\quad \mbox{for}\
k=1,2,\dots,6.
\end{equation*}

Second, we introduce the following estimates related to the vector fields.
\begin{lemma} 
	For any smooth function $m=m(t,x)$, the following estimates hold.
	\begin{enumerate}
		\item \emph{Estimate on commutators}. For all $I\in \mathbb{N}^{6}$, we have 
			\begin{equation}\label{est:parGamma}
			\begin{aligned}
		\sum_{\alpha=0}^{2}\left|\left[\partial_{\alpha},\Gamma^{I}\right]m\right|
		+\left|\left[S,\Gamma^{I}\right]m\right|&\lesssim \sum_{|J|<|I|}\sum_{\beta=0}^{2}\left|\partial_{\beta}\Gamma^{J}m\right|.
		\end{aligned}
		\end{equation}
		
		\item \emph{Estimates on $\partial m$ and $G_{i}m$.} We have 
		\begin{equation}\label{est:decaypm}
		\langle t-r\rangle\left|\partial m\right|+\langle t+r\rangle\left|G_{i}m\right|\lesssim \sum_{|I|=1}\left(\left|Sm\right|+\left|\Gamma^{I}m\right|\right).
		\end{equation}
	\end{enumerate}
\end{lemma}
\begin{proof}
	Proof of (i). Note that, for $k=1,2,\dots,6$ and $\alpha=0,1,2$, we have 
	\begin{equation*}
	\begin{aligned}
\left[\partial_{\alpha}, \Gamma_{k}\right]m,\  \left[S,\Gamma_{k}\right]m&\in \mbox{Span}\left\{\partial_{t}m,\partial_{1}m,\partial_{2}m\right\},
	\end{aligned}
	\end{equation*}
	which implies~\eqref{est:parGamma} for $I\in \mathbb{N}^{6}$ with $|I|=1$. Then by an induction argument, we obtain~\eqref{est:parGamma} for all $I\in \mathbb{N}^{6}$.
	
	Proof of (ii). By an elementary computation, for $i=1,2$,
	\begin{equation*}
	\begin{aligned}
	\partial_{t}m&=(t^{2}-r^{2})^{-1}\left(tSm-x^{j}H_{j}m\right),\\
	\partial_{i}m&=(t^{2}-r^{2})^{-1}\left(tH_{i}m-x_{i}Sm-(-1)^{i}x_{3-i}\Omega m\right),\\
	G_{i}m&=\frac{1}{r}\left(H_{i}m+(r-t)\partial_{i}m\right)=\frac{1}{t}\left(H_{i}m-\frac{x_{i}}{r}(r-t)\partial_{t}m\right).
	\end{aligned}
	\end{equation*}
	Based on the above identities, we obtain~\eqref{est:decaypm}.
\end{proof}

Third, we recall the following estimates related to the null form $Q$ from~\cite{Sogge}.
\begin{lemma}[\cite{Sogge}]
	For any $I\in \mathbb{N}^{6}$, and smooth functions $m$ and $n$, we have  
	\begin{align}
	\left|Q(m,n)\right|&\lesssim 
	\sum_{i=1,2}\left(\left|G_{i}m\right|\left|\partial n\right|+\left|G_{i}n\right|\left|\partial m\right|\right),\label{est:Q0decayG}\\
\left|Q(m,n)\right|&\lesssim 	\langle t+r\rangle^{-1}
	\sum_{|I|=1}\bigg(|Sm|+\sum_{|J|=1}\left|\Gamma^{J}m\right|\bigg)|\Gamma^{I}n|,\label{est:Q0decay}\\
		\left|\Gamma^{I}Q(m,n)\right|&\lesssim \sum_{\substack{ \alpha\ne \beta \\ |I_{1}|+|I_{2}|\le |I| }}\left(\left|Q_{0}\left(\Gamma^{I_{1}}m,\Gamma^{I_{2}}n\right)\right|+\left|Q_{\alpha\beta}\left(\Gamma^{I_{1}}m,\Gamma^{I_{2}}n\right)\right|\right)\label{equ:GammaQ0}.
	\end{align}
\end{lemma}

\begin{proof}
	The proof of~\eqref{est:Q0decayG} is based on the identity $\partial_{i}=G_{i}-\frac{x_{i}}{r}\partial_{t}$ and the structure of null form $Q$, and we omit it.
	For~\eqref{est:Q0decay} and~\eqref{equ:GammaQ0}, we refer to~\cite[Lemma 3.3]{Sogge} and~\cite[Page 58]{Sogge} for the complete proofs.
	\end{proof}

Note that, from~\eqref{est:Q0decay} and~\eqref{equ:GammaQ0}, for all $I\in \mathbb{N}^{6}$, we have the following pointwise estimate for $\Gamma^{I}Q$ where $Q\in \left\{Q_{0};\ Q_{\alpha\beta},\ 0\le \alpha\ne \beta\le 2\right\}$,
\begin{equation}\label{est:txQ0}
\left|\langle t+r\rangle\Gamma^{I}Q(w,v)\right|\lesssim \sum_{\substack{|I_{1}|+|I_{2}|\le |I|\\ |J_{1}|=|J_{2}|=1}}\bigg(\left|S\Gamma^{I_{1}}w\right|+\left|\Gamma^{J_{1}}\Gamma^{I_{1}}w\right|\bigg)\left|\Gamma^{J_{2}}\Gamma^{I_{2}}v\right|.
\end{equation}

Last, we expand the nonlinear term $C_{1}Q_{0}+C_{1}^{\alpha\beta}Q_{\alpha\beta}$ by an elementary computation.
\begin{lemma}\label{le:equQ0Qab}
 Let $(m,n)$ be a solution of~\eqref{equ:waveKG}, then we have 
 \begin{equation}\label{equ:Q0Qab}
C_{1}Q_{0}(m,n)+C^{\alpha\beta}_{1}Q_{\alpha\beta}(m,n)=\partial^{\alpha}F_{\alpha}(m,n)+\partial_{\alpha}H^{\alpha}(m,n)+G(m,n),
 \end{equation}
 where
 \begin{equation*}
 \begin{aligned}
 G(m,n)&=\frac{C^{2}_{1}}{2}n^{2}\left(C_{1}Q_{0}(m,n)+C_{1}^{\alpha\beta}Q_{\alpha\beta}(m,n)\right),\\
 F_{\alpha}(m,n)&=\sum_{k=1,2}\frac{C^{k}_{1}}{k!}\left(n^{k}\partial_{\alpha} m\right),\ H^{\alpha}(m,n)=\sum_{k=1,2}\frac{C_{1}^{k-1}}{k!}\left(C_{1}^{\beta\alpha}-C_{1}^{\alpha\beta}\right)(n^{k}\partial_{\beta}m).
 \end{aligned}
 \end{equation*}
\end{lemma}

\begin{proof}
	First, we observe that 
	\begin{equation*}
	C_{1}^{\alpha\beta}Q_{\alpha\beta}(m,n)=\left(C_{1}^{\beta\alpha}-C_{1}^{\alpha\beta}\right)\partial_{\alpha}(n\partial_{\beta}m).
	\end{equation*}
	Then, from~\eqref{equ:waveKG}, we have
	\begin{equation*}
	\begin{aligned}
	C_{1}Q_{0}(m,n)
	=C_{1}\partial^{\alpha}\left(n\partial_{\alpha}m\right)+C_{1}^{2}nQ_{0}(m,n)+C_{1}C_{1}^{\alpha\beta}nQ_{\alpha\beta}(m,n).
	\end{aligned}
	\end{equation*}
	By an elementary computation,
	\begin{equation*}
	\begin{aligned}
	C_{1}^{2}nQ_{0}(m,n)&=\frac{C_{1}^{2}}{2}\partial^{\alpha}\left(n^{2}\partial_{\alpha}m\right)+G(m,n),\\ C_{1}C_{1}^{\alpha\beta}nQ_{\alpha\beta}(m,n)&=\frac{C_{1}}{2}\left(C_{1}^{\beta\alpha}-C_{1}^{\alpha\beta}\right)\partial_{\alpha}(n^{2}\partial_{\beta}m).
	\end{aligned}
		\end{equation*}
		Based on the above identities, we obtain~\eqref{equ:Q0Qab}.
	\end{proof}
\begin{remark}
Note that, the terms $\partial^{\alpha}F_{\alpha}$ and $\partial_{\alpha}H^{\alpha}$ take the divergence form, and $G$ is a quartic term which decays fast enough as $t\to\infty$. This way of expressing the nonlinear terms can be used to establish very good $L^{2}$-norm estimate of wave component $w$, and hence a good estimate of the $L^{2}$-norm of $Sw$ by the aid of the conformal energy estimate (see more details in Lemma~\ref{le:enerE} and~\S\ref{Se:mainpro}).
\end{remark}

\subsection{Global Sobolev inequality}
In this subsection, we recall some Sobolev type inequalities associated with the vector field set $V$. These inequalities can be used to obtain the pointwise decay estimates of wave and Klein-Gordon equation from the weighted energy bounds.

\begin{lemma}[\cite{Geor,KlainWave}]
	Let $u=u(t,x)$ be a sufficiently regular function. Then the following estimates hold.
	\begin{enumerate}
		\item \emph{Standard Sobolev inequality}. We have 
		\begin{equation}\label{est:SobolevStand}
		\left|u(t,x)\right|\lesssim \langle r \rangle^{-\frac{1}{2}}\sum_{|I|\le 2}\left\|\Gamma^{I}u(t,x)\right\|_{L^{2}_{x}}.
		\end{equation}
		
		\item \emph{Estimate inside of a cone.} For $|x|\le \frac{t}{2}$, we have 
		\begin{equation}\label{est:Sobolevinsi}
		\left|u(t,x)\right|\lesssim \langle t\rangle^{-\frac{1}{2}}\sum_{|I|\le 3}\left\|\Gamma^{I}u(t,x)\right\|_{L^{2}_{x}}.
		\end{equation}
		
		\item \emph{Global Sobolev inequality.} We have 
		\begin{equation}\label{est:SobolevGlobal}
		\left|u(t,x)\right|\lesssim \langle t\rangle^{-\frac{1}{2}}\sum_{|I|\le 3}\left\|\Gamma^{I}u(t,x)\right\|_{L_{x}^{2}}.
		\end{equation}
	\end{enumerate}
\end{lemma}

\begin{proof}
	For~\eqref{est:SobolevStand} and~\eqref{est:Sobolevinsi}, we refer to~\cite[Proposition 1]{KlainWave} and~\cite[Lemma 2.4]{Geor} for the details of the proof respectively. Then, the inequality~\eqref{est:SobolevGlobal} is a consequence of~\eqref{est:SobolevStand} and~\eqref{est:Sobolevinsi}.
	\end{proof}
\begin{remark}
	The inequality~\eqref{est:SobolevGlobal} by Georgiev can be used to get the pointwise decay of a nice function with $t^{-\frac{1}{2}}$ rate which is the same as in the Klainerman-Sobolev inequality in 2D, but no bound on $\|S\Gamma^{I}u\|_{L_{x}^{2}}$ is required, which is more compatible with the coupled wave--Klein-Gordon equations. However, these two kinds of inequalities, only in the 2D case, yield the same decay rate of a function; and for higher dimensional case, the decay rate derived from Georgiev \cite{Geor} is slower than the rate derived from the Klainerman-Sobolev inequality (see more details in~\cite[Lemma 2.3 and Lemma 2.4]{Geor} and \cite[Proposition 3]{KlainWave}).
\end{remark}
\subsection{Estimates on 2D linear wave equation}

In this subsection, we recall several technical estimates for 2D linear wave equation. We start with the $L^{2}$ and $L^{\infty}$ estimates for 2D homogeneous wave equation. The proofs are similar to~\cite[Theorem 4.3.1 and Theorem 4.6.1]{LZBOOK}, but they are given for the sake of completeness and for the
readers’ convenience.
\begin{lemma}[\cite{LZBOOK}]\label{le:wave}
	Let $u$ be the solution to the Cauchy problem
	\begin{equation}\label{equ:WaveCauchyhom}
	\left\{\begin{aligned}
	-\Box u(t,x)&=0,\\
	(u,\partial_{t}u)_{|t=0}&=(u_{0},u_{1}).
	\end{aligned}\right.
	\end{equation}
	Then the following estimates hold.
	
	\begin{enumerate}
		\item \emph{$L^{2}$ estimate on solution.} We have 
		\begin{equation}\label{est:L2wave}
		\left\|u(t,x)\right\|_{L_{x}^{2}}\lesssim \left\|u_{0}\right\|_{L_{x}^{2}}+\log ^{\frac{1}{2}}(2+t)\left( \|u_{1}\|_{L_{x}^{1}}+\| u_{1}\|_{L^{2}_{x}}\right).
			\end{equation}
		
		\item  \emph{$L^{\infty}$ estimate on solution.} We have 
		\begin{equation}\label{est:Linwave}
		\|u(t,x)\|_{L_{x}^{\infty}}\lesssim \langle t\rangle^{-\frac{1}{2}}\left(\|u_{0}\|_{W^{2,1}}+\|u_{1}\|_{W^{1,1}}\right).
		\end{equation}
	\end{enumerate}
\end{lemma}

\begin{proof}
Proof of (i). Taking the Fourier transform in the Cauchy problem with respect to the argument $x$, we have
	\begin{equation*}
\left\{\begin{aligned}
&\partial_{t}^{2}\hat{u}(t,\xi)+|\xi|^{2}\hat{u}(t,\xi)=0,\\
&\hat{u}(0,\xi)=\hat{u}_{0}(\xi),\quad \partial_{t}\hat{u}(0,\xi)=\hat{u}_{1}(\xi).
\end{aligned}\right.
\end{equation*} 
We solve the above second-order ODE in $t$ to arrive at the expression of the solution $u$ in Fourier space
\begin{equation*}
\hat{u}(t,\xi)=\cos (t|\xi|)\hat{u}_{0}(\xi)+\frac{\sin (t|\xi|)}{|\xi|}\hat{u}_{1}(\xi),\quad \mbox{for}\ (t,\xi)\in [0,\infty)\times \R^{2}.
\end{equation*}
Note that, from the Plancherel theorem and polar coordinate transformation, 
\begin{equation*}
\begin{aligned}
\left\|\cos (t|\xi|)\hat{u}_{0}(\xi)\right\|_{L_{\xi}^{2}}
&\lesssim \|\hat{u}_{0}(\xi)\|_{L_{\xi}^{2}}\lesssim \|u_{0}\|_{L_{x}^{2}},\\
\left\|\frac{\sin (t|\xi|)}{|\xi|}\hat{u}_{1}(\xi)\textbf{1}_{\{|\xi|\ge 1\}}\right\|_{L_{\xi}^{2}}
&\lesssim \|\hat{u}_{1} (\xi)\|_{L_{\xi}^{2}} \left\|\frac{\sin (t|\xi|)}{|\xi|}\textbf{1}_{\{|\xi|\ge 1\}}\right\|_{L_{\xi}^{\infty}}\lesssim \|u_{1}\|_{L_{x}^{2}},\\
\left\|\frac{\sin (t|\xi|)}{|\xi|}\hat{u}_{1}(\xi)\textbf{1}_{\{|\xi|\le 1\}}\right\|_{L_{\xi}^{2}}
&\lesssim \|\hat{u}_{1} (\xi)\|_{L_{\xi}^{\infty}}\left(\int_{0}^{t}\frac{\sin^{2}r}{r}\d r\right)^{\frac{1}{2}}\lesssim \log^{\frac{1}{2}}(2+t)\|u_{1}\|_{L_{x}^{1}}.
\end{aligned}
\end{equation*}
Therefore, using again the Plancherel theorem, we obtain
\begin{equation*}
\begin{aligned}
\|u(t,x)\|_{L^{2}_{x}}=\|\hat{u}(t,\xi)\|_{L_{\xi}^{2}}
&\lesssim \left\|\cos (t|\xi|)\hat{u}_{0}(\xi)\right\|_{L^{2}_{\xi}}+\left\|\frac{\sin (t|\xi|)}{|\xi|}\hat{u}_{1}(\xi)\right\|_{L_{\xi}^{2}}\\
&\lesssim \|u_{0}\|_{L_{x}^{2}}+\|u_{1}\|_{L_{x}^{2}}+\log^{\frac{1}{2}}(2+t)\|u_{1}\|_{L_{x}^{1}},
\end{aligned}
\end{equation*}
which implies~\eqref{est:L2wave}.

Proof of (ii). From the expression of solutions for 2D linear wave equation, we decompose
\begin{equation*}
u(t,x)=\mathcal{I}_{1}(t,x)+\mathcal{I}_{2}(t,x)\quad \mbox{for}\ (t,x)\in [0,\infty)\times \R^{2},
\end{equation*} 
where
\begin{equation*}
\mathcal{I}_{1}(t,x)=\frac{1}{2\pi}\int_{B (x,t)}\frac{u_{1}(y)\d y}{\sqrt{t^{2}-|x-y|^{2}}},\quad 
\mathcal{I}_{2}(t,x)=\frac{1}{2\pi}\partial_{t}\int_{B (x,t)}\frac{u_{0}(y)\d y}{\sqrt{t^{2}-|x-y|^{2}}}.
\end{equation*}

\emph{Estimate on $\mathcal{I}_{1}$.} 
We claim that 
\begin{equation}\label{est:I1}
\left|\mathcal{I}_{1}(t,x)\right|\lesssim \langle t\rangle^{-\frac{1}{2}}\|u_{1}\|_{W^{1,1}},\quad \mbox{for}\ (t,x)\in [0,\infty)\times \R^{2}.
\end{equation}

Indeed, taking the change of variable $y\mapsto (x-y)$, we have 
\begin{equation}\label{equ:I1}
\mathcal{I}_{1}(t,x)=\frac{1}{2\pi}\int_{B (0,t)}\frac{u_{1}(x-y)}{\sqrt{t^{2}-|y|^{2}}}\d y,\quad \mbox{for}\ (t,x)\in [0,\infty)\times \R^{2}.
\end{equation}

Case I: Let $(t,x)\in [0,2]\times \R^{2}$. We rewrite~\eqref{equ:I1} as
\begin{equation*}
\mathcal{I}_{1}(t,x)=\frac{1}{2\pi}\int_{-t}^{t}\int_{-\sqrt{t^{2}-y^{2}_{1}}}^{\sqrt{t^{2}-y^{2}_{1}}}
\frac{u_{1}(x-(y_{1},y_{2}))}{\sqrt{t^{2}-y_{1}^{2}-y_{2}^{2}}}\d y_{2} \d y_{1}.
\end{equation*}
Note that, from the Newton-Leibniz formula and change of variable, we have 
\begin{equation*}
\begin{aligned}
\int_{-t}^{t}\left\|u(x-(y_{1},y_{2}))\right\|_{L_{y_{2}}^{\infty}}\d y_{1}
&\lesssim \int_{-t}^{t}\int_{\R}\left|\partial_{2}u_{1}(x-(y_{1},s))\right|\d s \d y_{1}\lesssim \|u_{1}\|_{W^{1,1}},\\
\max_{y_{1}\in [-t,t]}\int_{-\sqrt{t^{2}-y^{2}_{1}}}^{\sqrt{t^{2}-y^{2}_{1}}}\frac{\d y_{2}}{\sqrt{t^{2}-y^{2}_{1}-y^{2}_{2}}}
&= \max_{y_{1}\in [-t,t]}\int_{-1}^{1}\frac{\d s}{\sqrt{1-s^{2}}}\le 2\int_{0}^{1}\frac{\d s}{\sqrt{s}\sqrt{2-s}}\lesssim 1.
\end{aligned}
\end{equation*}
Combining the above estimates, we obtain
\begin{equation*}
\begin{aligned}
\left|\mathcal{I}_{1}(t,x)\right|\lesssim& \left(\int_{-t}^{t}\left\|u(x-(y_{1},y_{2}))\right\|_{L_{y_{2}}^{\infty}}\d y_{1}\right)\\
&\times \left(\max_{y_{1}\in [-t,t]}\int_{-\sqrt{t^{2}-y^{2}_{1}}}^{\sqrt{t^{2}-y^{2}_{1}}}\frac{\d y_{2}}{\sqrt{t^{2}-y^{2}_{1}-y^{2}_{2}}}\right)\lesssim \|u_{1}\|_{W^{1,1}},
\end{aligned}
\end{equation*}
which implies~\eqref{est:I1} for this case.

Case II: Let $(t,x)\in (2,\infty)\times \R^{2}$. We decompose
\begin{equation*}
\mathcal{I}_{1}(t,x)=\mathcal{I}_{11}(t,x)+\mathcal{I}_{12}(t,x)\quad \mbox{for}\ (t,x)\in (2,\infty)\times \R^{2},
\end{equation*}
where
\begin{equation*}
\begin{aligned}
\mathcal{I}_{11}(t,x)=\frac{1}{2\pi}\int_{B (0,{t-1})}\frac{u_{1}(x-y)\d y}{\sqrt{t^{2}-|y|^{2}}},\quad \mathcal{I}_{12}(t,x)=\frac{1}{2\pi}\int_{ B(0,t)\setminus B (0,{t-1})}\frac{u_{1}(x-y)\d y}{\sqrt{t^{2}-|y|^{2}}}.
\end{aligned}
\end{equation*}

First, we know that 
\begin{equation*}
\max_{y\in B(0,t-1)}\frac{1}{\sqrt{t^{2}-|y|^{2}}}=\max_{y\in B(0,t-1)}\frac{1}{\sqrt{t+|y|}\sqrt{t-|y|}}\lesssim \langle t\rangle^{-\frac{1}{2}},
\end{equation*}
which implies
\begin{equation}\label{est:I11}
\left|\mathcal{I}_{11}(t,x)\right|\lesssim \int_{\R^{2}}|u(y)|\d y\left(\max_{y\in B(0,t-1)}\frac{1}{\sqrt{t^{2}-|y|^{2}}}\right)\lesssim \langle t\rangle^{-\frac{1}{2}}\|u_{1}\|_{L_{x}^{1}}.
\end{equation}
Second, by the polar coordinate transformation $y=\rho\omega$ with $\d y=\rho\d \rho\d \omega$,
\begin{align*}
\mathcal{I}_{12}(t,x)
&=\frac{1}{2\pi}\int_{\partial B(0,1)}\int_{t-1}^{t}\frac{\rho u_{1}(x-\rho\omega)}{\sqrt{t^{2}-\rho^{2}}}\d \rho \d \omega\\
&=-\frac{1}{\pi}\int_{\partial B(0,1)}\int_{t-1}^{t}\frac{\rho u_{1}(x-\rho\omega)}{\sqrt{t+\rho}}\d \sqrt{t-\rho}\d \omega.
\end{align*}
Based on the above identity and integration by parts, we decompose
\begin{equation*}
\mathcal{I}_{12}(t,x)=\mathcal{I}_{12}^{1}(t,x)+\mathcal{I}_{12}^{2}(t,x),
\end{equation*}
where
\begin{align*}
\mathcal{I}_{12}^{1}(t,x)&=\frac{1}{\pi}\int_{\partial B(0,1)}\frac{(t-1)u_{1}(x-(t-1)\omega)}{\sqrt{2t-1}}\d \omega,\\
\mathcal{I}_{12}^{2}(t,x)&=\frac{1}{\pi}\int_{\partial B(0,1)}\int_{t-1}^{t}\sqrt{t-\rho}\partial_{\rho}\left(\frac{\rho u_{1}(x-\rho\omega)}{\sqrt{t+\rho}}\right)\d \rho \d \omega.
\end{align*}

Using again the Newton-Leibniz formula, we have 
\begin{equation*}
(t-1)u_{1}(x-(t-1)\omega)=-\int_{t-1}^{\infty}\left(u_{1}(x-\rho\omega)-\rho\omega\cdot \nabla u_{1}(x-\rho\omega)\right)\d \rho,
\end{equation*}
which implies
\begin{align*}
\left|\mathcal{I}_{12}^{1}(t,x)\right|
&\lesssim (2t-1)^{-\frac{1}{2}}\int_{\partial B(0,1)}\int_{t-1}^{\infty}\left(|u_{1}(x-\rho\omega)|+|\rho\omega\cdot \nabla u_{1}(x-\rho\omega)|\right)\d \rho\d \omega\\
&\lesssim \langle t\rangle^{-\frac{1}{2}} \int_{\R^{2}\setminus B(x,t-1)}\left(\frac{|u_{1}(y)|}{|x-y|}+|\nabla u_{1}(y)|\right)\d y\lesssim \langle t\rangle^{-\frac{1}{2}}\|u_{1}\|_{W^{1,1}}.
\end{align*}
On the other hand, by direct computation, we have 
\begin{equation*}
\partial_{\rho}\left(\frac{\rho u_{1}(x-\rho\omega)}{\sqrt{t+\rho}}\right)
=\frac{u_{1}(x-\rho\omega)}{\sqrt{t+\rho}}-\rho\omega\cdot\frac{\nabla u_{1}(x-\rho\omega)}{\sqrt{t+\rho}}-\frac{\rho u_{1}(x-\rho\omega)}{2(t+\rho)^{\frac{3}{2}}},
\end{equation*}
which implies
\begin{align*}
\left|\mathcal{I}_{12}^{2}(t,x)\right|
&\lesssim (2t-1)^{-\frac{1}{2}}
\int_{\partial B(0,1)}\int_{t-1}^{\infty}\left(|u_{1}(x-\rho\omega)|+\rho|\nabla u_{1}(x-\rho\omega)|\right)\d \rho \d \omega \\
&\lesssim \langle t\rangle^{-\frac{1}{2}} \int_{\R^{2}\setminus B(x,t-1)}\left(\frac{|u_{1}(y)|}{|x-y|}+|\nabla u_{1}(y)|\right)\d y\lesssim \langle t\rangle^{-\frac{1}{2}}\|u_{1}\|_{W^{1,1}}.
\end{align*}

Combining the above two estimates, we obtain
\begin{equation}\label{est:I12}
\left|\mathcal{I}_{12}(t,x)\right|\lesssim \left|\mathcal{I}_{12}^{1}(t,x)\right|+\left|\mathcal{I}_{12}^{2}(t,x)\right|\lesssim \langle t\rangle^{-\frac{1}{2}}\|u_{1}\|_{W^{1,1}}.
\end{equation}
We see that~\eqref{est:I1} for this case follows from~\eqref{est:I11} and~\eqref{est:I12}.

\smallskip
\emph{Estimate on $\mathcal{I}_{2}$.}
We claim that
\begin{equation}\label{est:I2}
\left|\mathcal{I}_{2}(t,x)\right|\lesssim \langle t\rangle^{-\frac{1}{2}}\|u_{0}\|_{W^{2,1}},\quad \mbox{for}\ (t,x)\in [0,\infty)\times \R^{2}.
\end{equation}
Indeed, taking the change of variable $y\mapsto x-ty$, we have 
\begin{equation*}
\mathcal{I}_{2}(t,x)=\frac{1}{2\pi}\partial_{t}\int_{B(0,1)}\frac{tu_{0}(x-ty)}{\sqrt{1-|y|^{2}}}\d y=\mathcal{I}_{21}(t,x)+\mathcal{I}_{22}(t,x),
\end{equation*}
where
\begin{align*}
\mathcal{I}_{21}(t,x)&=\frac{1}{2\pi}\int_{B(0,1)}\frac{u_{0}(x-ty)}{\sqrt{1-|y|^{2}}}\d y,\quad \quad \ \quad \mbox{for}\ (t,x)\in [0,\infty)\times \R^{2},\\
\mathcal{I}_{22}(t,x)&=-\frac{1}{2\pi}\int_{B(0,t)}\frac{y\cdot \nabla u_{0}(x-y)}{t\sqrt{t^{2}-|y|^{2}}}\d y,\quad \mbox{for}\ (t,x)\in [0,\infty)\times \R^{2}.
\end{align*}
Using a similar argument as in the proof of~\eqref{est:I1}, one can obtain~\eqref{est:I2}.

Combining~\eqref{est:I1} and~\eqref{est:I2}, we complete the proof of~\eqref{est:Linwave}.
	\end{proof}

Second, we introduce the $L^{2}$ and $L^{\infty}$ estimates of solutions for the 2D inhomogeneous wave equation with zero initial data. These estimates will be used to control a part of wave component $w$ which is related to quartic term $G(w,v)$ (recall Lemma \ref{le:equQ0Qab} for the expression of $G$).
\begin{lemma}
	Let $u$ be the solution to the Cauchy problem
\begin{equation*}
\left\{\begin{aligned}
-\Box u(t,x)&=f(t,x),\\
(u,\partial_{t}u)_{|t=0}&=(0,0),
\end{aligned}\right.
\end{equation*}
with $f(t,x)$ a sufficiently regular function. Then the following estimates hold.
\begin{enumerate}
	\item \emph{$L^{2}$ estimate on solution.} We have 
	\begin{equation}\label{est:L2waveinh}
	\|u(t,x)\|_{L_{x}^{2}}\lesssim \log^{\frac{1}{2}}(2+t)\int_{0}^{t}\left(\|f(s,x)\|_{L_{x}^{1}}+\|f(s,x)\|_{L_{x}^{2}}\right)\d s.
	\end{equation}
	
	\item \emph{$L^{\infty}$ estimate on solution.} We have 
	\begin{equation}\label{est:Liniwaveinh}
	\begin{aligned}
	\left\|u(t,x)\right\|_{L_{x}^{\infty}}
	&\lesssim \langle t\rangle^{-\frac{1}{2}}\int_{0}^{t}(1+s)^{-\frac{1}{2}}\left\|Sf(s,x)\right\|_{L_{x}^{1}}\d s \\
	&+\langle t\rangle^{-\frac{1}{2}}\sum_{|I|\le 1}\int_{0}^{t}(1+s)^{-\frac{1}{2}}\left\|\Gamma^{I}f(s,x)\right\|_{L_{x}^{1}}\d s.
	\end{aligned}
	\end{equation}
\end{enumerate}

\end{lemma}

\begin{proof}
 Proof of (i). Taking the Fourier transform for the Cauchy problem and using the Duhamel's principle,
 \begin{equation*}
 \hat{u}(t,\xi)=\mathcal{I}_{3}(t,\xi)+\mathcal{I}_{4}(t,\xi),\quad t\in [0,\infty),
 \end{equation*}
 where
 \begin{align*}
 &\mathcal{I}_{3}(t,\xi)=\textbf{1}_{\{|\xi|\ge 1\}}\int_{0}^{t}\frac{\sin ((t-s)|\xi|)}{|\xi|}\hat{f}(s,\xi)\d s,\quad t\in [0,\infty),\\
 &\mathcal{I}_{4}(t,\xi)=\textbf{1}_{\{|\xi|\le 1\}}\int_{0}^{t}\frac{\sin ((t-s)|\xi|)}{|\xi|}\hat{f}(s,\xi)\d s,\quad t\in [0,\infty).
 \end{align*}
 Note that, from the Plancherel theorem
 \begin{align*}
 &\left\|\mathcal{I}_{3}(t,\xi)\right\|_{L_{\xi}^{2}}\lesssim \int_{0}^{t}\|\hat{f}(s,\xi)\|_{L_{\xi}^{2}}\left\|\frac{\sin (t-s)|\xi|}{|\xi|}\textbf{1}_{\{|\xi|\ge 1\}}\right\|_{L_{\xi}^{\infty}}\d s \lesssim\int_{0}^{t} \|{f}(s,x)\|_{L_{x}^{2}}\d s,\\
  &\left\|\mathcal{I}_{4}(t,\xi)\right\|_{L_{\xi}^{2}}\lesssim \int_{0}^{t}\|\hat{f}(s,\xi)\|_{L_{\xi}^{\infty}}\left(\int_{0}^{t}\frac{\sin^{2} r}{r}\d r\right)^{\frac{1}{2}}\d s \lesssim \log^{\frac{1}{2}}(2+ t)\int_{0}^{t}\|{f}(s,x)\|_{L_{x}^{1}}\d s.
 \end{align*}
Therefore, using again the Plancherel theorem, we have 
\begin{align*}
\|u(t,x)\|_{L_{x}^{2}}=\|\hat{u}(t,\xi)\|_{L_{\xi}^{2}}
&\lesssim 
\|\mathcal{I}_{3}(t,\xi)\|_{L_{\xi}^{2}}+\|\mathcal{I}_{4}(t,\xi)\|_{L_{\xi}^{2}}\\
&\lesssim \int_{0}^{t}\|f(s,x)\|_{L^{2}_{x}}\d s+\log^{\frac{1}{2}}(2+t) \int_{0}^{t}\|f(s,x)\|_{L_{x}^{1}}\d s,
\end{align*}
which implies~\eqref{est:L2waveinh}.

Proof of (ii). Estimate~\eqref{est:Liniwaveinh} is due to H\"ormander~\cite{Hor}. We also refer to~\cite[Theorem 4.6.2]{LZBOOK} for details of the proof.
	\end{proof}

To establish the energy estimates of the wave equation for future reference, we first introduce the standard energy $\mathcal{E}$ and the conformal energy $\mathcal{G}$ for the 2D wave equation,
\begin{equation*}
\begin{aligned}
\mathcal{E}(t,u)&=\int_{\R^{2}}\left((\partial_{t}u)^{2}+(\partial_{1}u)^{2}+(\partial_{2}u)^{2}\right)(t,x)\d x,\\
\mathcal{G}(t,u)&=\int_{\R^{2}}\bigg((Su+u)^{2}+(\Omega u)^{2}+\sum_{i=1,2}({H}_{i}u)^{2}\bigg)(t,x)\d x.
\end{aligned}
\end{equation*}

Now we recall the energy estimates for 2D wave equation from~\cite{Alin,Sogge}. 
\begin{lemma}[\cite{Alin,Sogge}]\label{le:enerE}
	Let $u$ be the solution to the Cauchy problem
		\begin{equation*}
	\left\{\begin{aligned}
	-\Box u(t,x)&=f(t,x),\\
	(u,\partial_{t}u)_{|t=0}&=(u_{0},u_{1}),
	\end{aligned}\right.
	\end{equation*}
	with $f(t,x)$ a sufficiently regular function. Then the following estimates hold.
	\begin{enumerate}
			\item \emph{Standard energy estimate.} We have 
		\begin{equation}\label{est:EnerE}
		\mathcal{E} (t,u)^{\frac{1}{2}}\lesssim 
		\mathcal{E}(0,u)^{\frac{1}{2}}+\int_{0}^{t}\|f(s,x)\|_{L_{x}^{2}}\d s.
		\end{equation}
		
		\item \emph{Conformal energy estimate.} We have 
		\begin{equation}\label{est:ConE}
		\mathcal{G}(t,u)^{\frac{1}{2}}\lesssim \mathcal{G}(0,u)^{\frac{1}{2}}+\int_{0}^{t}\|\langle s+r \rangle f(s,x)\|_{L_{x}^{2}}\d s.
		\end{equation}
		
	\end{enumerate}
\end{lemma}

\begin{proof}
	Proof of (i). First, we can rewrite the product $\left(-\Box u\right)\partial_{t}u$ as the following divergence form,
	\begin{equation*}
	\left(-\Box u\right)\partial_{t}u=\frac{1}{2}\sum_{\alpha=0}^{2}\partial_{t}(\partial_{\alpha}u)^{2}-\partial^{i}\left(\partial_{i}u\partial_{t}u\right).
		\end{equation*}
		Integrating the above identity on $[0,t]\times \R^{2}$ for any $t>0$, and using $-\Box u=f$
		\begin{equation*}
		\begin{aligned}
		\mathcal{E}(t,u)
		&\le\mathcal{E}(0,u)+2\int_{0}^{t}\left\|f(s,x)\right\|_{L_{x}^{2}}\left\|\partial_{t}u(s,x)\right\|_{L_{x}^{2}}\d s\\
		&\le \mathcal{E}(0,u)+2\int_{0}^{t}\left\|f(s,x)\right\|_{L_{x}^{2}}\mathcal{E}(s,u)^{\frac{1}{2}}\d s.
		\end{aligned}
		\end{equation*}
		Based on the above inequality, we obtain~\eqref{est:EnerE}.
	
	Proof of (ii). Here, we briefly sketch the proof and refer to \cite[Theorem 6.11]{Alin} for the complete proof.
	Consider the nonspacelike multiplier
	\begin{equation*}
	K_{0}=(r^{2}+t^{2})\partial_{t}+2rt\partial_{r}\quad \mbox{where}\ \partial_{r}=\frac{x^{i}}{r}\partial_{i}.
	\end{equation*}
	By an elementary computation, we have 
	\begin{equation*}
	\begin{aligned}
	-\Box u\left(K_{0}u+tu\right)
	&=\frac{1}{2}\partial_{t}\bigg((Su+u)^{2}+(\Omega u)^{2}+({H}_{1}u)^{2}+({H}_{2}u)^{2}\bigg)\\
	&+\partial^{i}\left(tx_{i}\left(-(\partial_{t}u)^{2}+(\partial_{1}u)^{2}+(\partial_{2}u)^{2}\right)-tu\partial_{i}u\right)\\
	&-\partial^{i}\left((r^{2}+t^{2})(\partial_{t}u)(\partial_{i}u)+2rt(\partial_{r}u)(\partial_{i}u)+ {\frac{1}{2}\partial_t}(x_{i}u^{2})\right).
	\end{aligned}
	\end{equation*} 
	Integrating the above identity on $[0,t]\times \R^{2}$ and then using the Cauchy-Schwarz inequality, we see that 
	\begin{equation*}
	\mathcal{G}(t,u)-\mathcal{G}(0,u)\lesssim \int_{0}^{t}\left\|\langle s+r\rangle f(s,x)\right\|_{L_{x}^{2}}\left\|\langle s+r\rangle^{-1}\left(K_{0}u+tu\right)\right\|_{L_{x}^{2}}\d s.
	\end{equation*}
	From the definition of $K_{0}$, we observe that 
	\begin{equation*}
	K_{0}u=tSu+r(t\partial_{r}u+r\partial_{t}u)\Rightarrow\left|K_{0}u+tu\right|\lesssim \langle t+r\rangle \left(|Su+u|+|H_{1}u|+|H_{2}u|\right).
	\end{equation*}
	Combining the above inequalities, we obtain~\eqref{est:ConE}.
	\end{proof}

Last, we recall the extra decay for Hessian of 2D inhomogeneous wave equation. 
For the convenience of the reader, we revisit the complete proof in ~\cite{LEMA}.

\begin{lemma}[Extra decay for Hessian]\label{lem:Hessian}
	Let $u$ be the solution of the Cauchy problem
		\begin{equation*}
	\left\{\begin{aligned}
	-\Box u(t,x)&=f(t,x),\\
	(u,\partial_{t}u)_{|t=0}&=(u_{0},u_{1}).
	\end{aligned}\right.
	\end{equation*}
	Then we have 
	\begin{equation}\label{est:extwaveHe}
	\left|\partial \partial u\right|\lesssim \sum_{|I|=0,1}\frac{1}{\langle t-r\rangle}\left|\partial \Gamma^{I} u\right|+\frac{t}{\langle t-r\rangle}|f|,\quad \mbox{for}\ r\le 2t.
	\end{equation}
\end{lemma}

\begin{proof}
Case I: Let $t\in [0,1]$. It is easily seen that 
\begin{equation}\label{est:ext1}
\left|\partial \partial u\right|\lesssim \sum_{|I|=0,1}\left|\partial \Gamma^{I} u\right|\lesssim
\sum_{|I|=0,1}\frac{1}{\langle t-r\rangle}\left|\partial \Gamma^{I} u\right|+\frac{t}{\langle t-r\rangle}|f|,\quad \mbox{for}\ r\le 2t.
\end{equation}
Case II: Let $t\in (1,\infty)$. We first express the wave operator $-\Box$ by $\partial_{\alpha}$ and $H_{i}$ to obtain
\begin{align*}
-\Box=\frac{(t-r)(t+r)}{t^{2}}\partial_{t}\partial_{t}+\frac{x^{i}}{t^{2}}\partial_{t}H_{i}-\frac{1}{t}\partial^{i}H_{i}+\frac{2}{t}\partial_{t}-\frac{x^{i}}{t^{2}}\partial_{i}.
\end{align*}
Based on the above identity and $-\Box u=f$, we have 
\begin{equation}\label{est:ext2tt}
\left|\partial_{t}\partial_{t}u\right|\lesssim \sum_{|I|=0,1}\frac{1}{\langle r-t\rangle}\left|\partial \Gamma^{I}u\right|+\frac{t^{2}}{\langle r-t \rangle \langle r+t\rangle}|f|,\quad \mbox{for}\ r\le 2t.
\end{equation}
On the other hand, for $i,j\in \{1,2\}$, we have the following two identities
\begin{align*}
\partial_{i}\partial_{t}&={-\frac{x_{i}}{t}}\partial_{t}\partial_{t}+\frac{1}{t}\partial_{t}{H}_{i}-\frac{1}{t}\partial_{i},\\
\partial_{i}\partial_{j}&=\frac{x_{i}x_{j}}{t^{2}}\partial_{t}\partial_{t}-\frac{x_{i}}{t^{2}}\partial_{t}H_{j}+\frac{1}{t}\partial_{j}H_{i}-\frac{\delta_{ij}}{t}\partial_{t}+\frac{x_{i}}{t^{2}}\partial_{j}.
\end{align*}
Based on the above two identities and~\eqref{est:ext2tt}, we have 
\begin{equation}\label{est:ext2tx}
|\partial_{i}\partial_{t}u|+|\partial_{i}\partial_{j}u|\lesssim
\sum_{|I|=0,1}\frac{1}{\langle r-t\rangle}\left|\partial \Gamma^{I}u\right|+\frac{t^{2}}{\langle r-t \rangle \langle r+t\rangle}|f| ,\quad \mbox{for}\ r\le 2t.
\end{equation}
We see that~\eqref{est:extwaveHe} follows from~\eqref{est:ext1},~\eqref{est:ext2tt} and~\eqref{est:ext2tx}.
	\end{proof}

\begin{remark}
	The above Lemma states that, for the solution $u$ of wave equation, the Hessian form $\partial\partial u$ has extra $\langle t-r\rangle^{-1}$ decay than $\partial u$ in the spacetime region $\left\{(t,x):r\le 2t\right\}$ if the source term $f$ has sufficiently fast decay. Note that, this extra decay can be used to obtain the sharp pointwise decay for $\partial w$ thanks to the hidden divergence form structure in the wave equation of $w$ (see more details in \S\ref{Se:mainpro}).
\end{remark}

\subsection{Estimates on 2D linear Klein-Gordon equation}

In this subsection, we recall some decay estimates and energy estimates for 2D linear Klein-Gordon equation. First, we recall the following decay estimates from~\cite{Geor}.

\begin{theorem}[\cite{Geor}]\label{thm:KGdecay}
	Let $u$ be the solution of the Cauchy problem
		\begin{equation*}
	\left\{\begin{aligned}
	\left(-\Box+1\right) u(t,x)&=f(t,x),\\
	(u(0,x),\partial_{t}u(0,x))&=(u_{0}(x),u_{1}(x)),
	\end{aligned}\right.
	\end{equation*}
	with $f(t,x)$ a sufficiently regular function. Then we have 
	\begin{equation}\label{est:KGLini}
	\begin{aligned}
	\langle t+r\rangle|u(t,x)|
	&\lesssim \sum_{j=0}^{\infty}\sum_{|I|\le 5}\left\|\langle |x|\rangle \psi_{j}(|x|)\Gamma^{I}u(0,x)\right\|_{L_{x}^{2}}\\
	&+\sum_{j=0}^{\infty}\sum_{|I|\le 4}\max_{0\le s\le t}\psi_{j}(s)\left\|\langle s+|x|\rangle\Gamma^{I}f(s,x)\right\|_{L_{x}^{2}}.
	\end{aligned}
	\end{equation}
\end{theorem}

As a consequence, we have the following simplified version of Theorem~\ref{thm:KGdecay}.
\begin{corollary}\label{coro:KGdecay}
	With the same settings as Theorem~\ref{thm:KGdecay}, let $\delta_{0}>0$ and assume 
	\begin{equation}\label{est:assumf}
	\sum_{|I|\le 4}\max_{0\le s\le t}\langle s\rangle^{\delta_{0}}\|\langle s+|x|\rangle \Gamma^{I}f(s,x)\|_{L_{x}^{2}}\le C_{f},
	\end{equation}
	then we have 
	\begin{equation}\label{est:KGLINI}
	\langle t+r\rangle |u(t,x)|\lesssim \frac{C_{f}}{1-2^{-\delta_{0}}}+\sum_{|I|\le 5}\left\|\langle |x|\rangle \log (2+|x|)\Gamma^{I}u(0,x)\right\|_{L_{x}^{2}}.
	\end{equation}
\end{corollary}

\begin{proof}
	From $\mbox{supp}\psi_{0}\cap\R_{+}=[0,2]$ and $\mbox{supp}\psi_{j}=[2^{j-1},2^{j+1}]$ for $j\ge 1$, we infer
		\begin{align*}
	&\sum_{|I|\le 5}\left\|\langle |x|\rangle \psi_{j}( |x| )\Gamma^{I}u(0,x)\right\|_{L_{x}^{2}}\\
	&\lesssim \sum_{|I|\le 5}\left\|\langle x\rangle \log (2+|x|)\psi^{\frac{1}{2}}_{j}(| x|)\Gamma^{I}u(0,x)\right\|_{L_{x}^{2}}\left\|\frac{\psi_{j}^{\frac{1}{2}}(|x|)}{\log (2+|x|)}\right\|_{L^{\infty}_{x}}\\
	&\lesssim \frac{1}{(j+1)}\sum_{|I|\le 5}\left\|\langle |x|\rangle \log (2+|x|)\psi^{\frac{1}{2}}_{j}(| x|)\Gamma^{I}u(0,x)\right\|_{L_{x}^{2}}\quad \mbox{for}\ j\ge 0.
	\end{align*}
	Based on the above estimates, the Cauchy-Schwarz inequality and $\sum\psi_{j}=1$, 
	\begin{equation}\label{est:ini}
	\sum_{j=0}^{\infty}\sum_{|I|\le 5}\left\|\langle |x|\rangle \psi_{j}( |x| )\Gamma^{I}u(0,x)\right\|_{L_{x}^{2}}\lesssim \sum_{|I|\le 5}\left\|\langle |x|\rangle \log (2+|x|)\Gamma^{I}u(0,x)\right\|_{L_{x}^{2}}.
	\end{equation}
	On the other hand, using again the definition of $\psi_{j}$ and~\eqref{est:assumf}, for $j\ge 0$, we have 
	\begin{equation*}
	\max_{0\le s\le t}\psi_{j}(s)\left\|\langle s+|x|\rangle\Gamma^{I}f(s,x)\right\|_{L^{2}_{x}}\lesssim C_{f}\max_{0\le s\le t}\psi_{j}(s)\langle s\rangle^{-\delta_{0}}\lesssim C_{f}2^{-j\delta_{0}},
	\end{equation*}
	which implies
	\begin{equation}\label{est:f}
	\sum_{j=0}^{\infty}\sum_{|I|\le 4}	\max_{0\le s\le t}\psi_{j}(s)\left\|\langle s+|x|\rangle\Gamma^{I}f(s,x)\right\|_{L^{2}_{x}}\lesssim C_{f}\sum_{j=0}^{\infty}2^{-j\delta_{0}}\lesssim \frac{C_{f}}{1-2^{\delta_{0}}}.
	\end{equation}
	We see that~\eqref{est:KGLINI} follows from~\eqref{est:ini} and~\eqref{est:f}.
	\end{proof}

Similar to the case of 2D linear wave equation, we introduce the following standard energy $\mathcal{E}_{1}$ for the 2D linear Klein-Gordon equation,

\begin{equation*}
\mathcal{E}_{1}(t,u)=\int_{\R^{2}}\left((\partial_{t}u)^{2}+(\partial_{1}u)^{2}+(\partial_{2}u)^{2}+u^{2}\right)(t,x)\d x.
\end{equation*}

Now we introduce the energy estimates for 2D Klein-Gordon equation.

\begin{lemma}
Let $u$ be the solution to the Cauchy problem
\begin{equation*}
\left\{\begin{aligned}
\left(-\Box+1\right) u(t,x)&=f(t,x),\\
(u(0,x),\partial_{t}u(0,x))&=(u_{0}(x),u_{1}(x)),
\end{aligned}\right.
\end{equation*}
with $f(t,x)$ a sufficiently regular function. Then the following estimates hold.
	\begin{enumerate}
		\item \emph{Standard energy estimate.} We have 
		\begin{equation}\label{est:EnergK}
		\mathcal{E}_{1}(t,u)^{\frac{1}{2}}\lesssim 
		\mathcal{E}_{1}(0,u)^{\frac{1}{2}}+\int_{0}^{t}\|f(s,x)\|_{L_{x}^{2}}\d s.
		\end{equation}
		
		\item \emph{Ghost weight estimate (see also \cite{Alinhac01b}).} For all $\delta_{0}>0$, we have 
		\begin{equation}\label{est:GhostGu}
		\begin{aligned}
&	\sum_{i=1,2}\int_{0}^{t}\langle s \rangle^{-\delta_{0}}\int_{\R^{2}}\left(\frac{u^{2}}{\langle r-s\rangle^{\frac{3}{2}}}+\frac{|G_{i}u|^{2}}{\langle r-s\rangle^{\frac{3}{2}}}\right)\d x \d s\\
	&\lesssim \mathcal{E}_1(0,u)+\int_{0}^{t}\langle s\rangle^{-\delta_{0}}\|f(s,x)\|_{L_{x}^{2}}\|\partial_{t}u(s,x)\|_{L^{2}_{x}}\d s.
	\end{aligned}
		\end{equation}
		
	\end{enumerate}
\end{lemma}

\begin{proof}
	Proof of (i). The proof is similar to Lemma~\ref{le:enerE} (i), and we omit it.

Proof of (ii). Set 
\begin{equation*}
q(t,r)=\int_{-\infty}^{r-t}\langle s\rangle^{-\frac{3}{2}}\d s\quad \mbox{for}\ (t,r)\in \R_{+}\times \R_{+}.
\end{equation*}
We apply the multiplier $\langle t\rangle^{-\delta_{0}}e^{q}\partial_{t}u$ and obtain:
\begin{equation*}
\begin{aligned}
&\langle t\rangle^{-\delta_{0}}e^{q}\partial_{t}u\left(-\Box u+u\right)\\
&=\frac{1}{2}\partial_{t}\bigg(\langle t\rangle^{-\delta_{0}}e^{q}\bigg(\sum_{\alpha=0}^{3}(\partial_{\alpha} u)^{2}+u^{2}\bigg)\bigg)-\partial^{i}\left(\langle t\rangle^{-\delta_{0}}e^{q}\partial_{t}u\partial_{i}u\right)\\
&+\frac{1}{2}\frac{\langle t\rangle^{-\delta_{0}}e^{q}}{\langle t-r\rangle^{\frac{3}{2}}}\left((G_{1}u)^{2}+(G_{2}u)^{2}+u^{2}\right)+\frac{\delta_{0}}{2}t\langle t\rangle^{-(\delta_{0}+2)}e^{q}\bigg(\sum_{\alpha=0}^{3}(\partial_{\alpha} u)^{2}+u^{2}\bigg).
\end{aligned}
\end{equation*}
Integrating the above identity on $[0,t]\times \R^{2}$ and then using the Cauchy-Schwarz inequality, we obtain~\eqref{est:GhostGu}.
	\end{proof}

\section{Proof of Theorem~\ref{thm:main}}\label{sec:proof}
In this section, we prove the existence of global-in-time solution $(w,v)$ of~\eqref{equ:waveKG} satisfying~\eqref{est:thm} in Theorem~\ref{thm:main}. The proof relies on a bootstrap argument of high-order energy and pointwise decay of solutions.

\subsection{Bootstrap setting} 
Fix $0<\delta\ll 1$. The proof of Theorem~\ref{thm:main} is based on the following bootstrap setting: for $C_{0}\gg 1$ and $0<\varepsilon\ll C_{0}^{-1}$ to be chosen later
\begin{equation}\label{est:Bootw}
\left\{\begin{aligned}
\langle t\rangle^{\frac{1}{2}}\|w\|_{L_{x}^{\infty}}+\sum_{|I|\le N-1}\mathcal{E}(t,\Gamma^{I}w)^{\frac{1}{2}}+\sum_{|I|\le N}\langle t\rangle^{-\delta}\mathcal{E}(t,\Gamma^{I}w)^{\frac{1}{2}}&\le C_{0}\varepsilon,\\
\sum_{|I|\le N-1}\langle t\rangle^{-\frac{1}{2}-2\delta}\|S\Gamma^{I}w\|_{L_{x}^{2}}+\sum_{|I|\le N-6}\langle t\rangle^{-\delta}\|S\Gamma^{I}w\|_{L_{x}^{2}}&\le C_{0}\varepsilon ,\\
\sum_{|I|\le N}\langle t\rangle^{-\delta}\|\Gamma^{I}w\|_{L^{2}_{x}}+\sum_{|I|\le N-9}\sup_{x\in \R^{2}}\langle t-r\rangle^{\frac{3}{4}}\langle t\rangle^{\frac{1}{2}}\left|\partial \Gamma^{I}w(t,x)\right|&\le C_{0}\varepsilon,
\end{aligned}\right.
\end{equation}
\begin{equation}\label{est:Bootv}
\left\{\begin{aligned}
\sum_{|I|\le N-5}\sup_{x\in \R^{2}}\langle t+r\rangle \left|\Gamma^{I}v(t,x)\right|&\le C_{0}\varepsilon,\\
\sum_{|I|\le N-1}\mathcal{E}_{1}(t,\Gamma^{I}v)^{\frac{1}{2}}+\sum_{|I|\le N}\langle t\rangle^{-\delta}\mathcal{E}_{1}(t,\Gamma^{I}v)^{\frac{1}{2}}&\le C_{0}\varepsilon,\\
\sum_{|I|\le N}\langle t\rangle^{-\delta}\int_{0}^{t}\langle s \rangle^{-\delta}\int_{\R^{2}}\frac{\left(\Gamma^{I}v\right)^{2}}{\langle r-s\rangle^{\frac{3}{2}}}+\frac{(G_{i}\Gamma^{I}v)^{2}}{\langle r-s\rangle^{\frac{3}{2}}}\d x \d s&\le C_{0}\varepsilon.
\end{aligned}\right.
\end{equation}
Let $(\vec{w}_{0},\vec{v}_{0})$ satisfy~\eqref{est:Bootw} and~\eqref{est:Bootv} at $t=0$. Let $(w,v)$ be the  corresponding solution of~\eqref{equ:waveKG} and define
\begin{equation}\label{def:T}
T_{*}(\vec{w}_{0},\vec{v}_{0})=\sup\left\{t\in[0,\infty):(w,v)\ \mbox{satisfies}~\eqref{est:Bootw}\ \mbox{and}~\eqref{est:Bootv}\ \mbox{on}\ [0,t]\right\}.
\end{equation}
The following proposition is the main part of the proof of Theorem~\ref{thm:main}.
\begin{proposition}\label{prop:main}
	For all initial data $(\vec{w}_{0},\vec{v}_{0})$ satisfying~\eqref{est:small} in Theorem~\ref{thm:main}, we have $T_{*}(\vec{w}_{0},\vec{v}_{0})=\infty$.
\end{proposition}
Note that, combining the bootstrap setting~\eqref{est:Bootw} and~\eqref{est:Bootv}, Theorem~\ref{thm:main} is a consequence of Proposition~\ref{prop:main}. The rest of the section is devoted to the proof of Proposition~\ref{prop:main}. From now on, the implied constants in $\lesssim$ do not depend on the constants $C_{0}$ and $\varepsilon$ appearing in the bootstrap assumption~\eqref{est:Bootw} and~\eqref{est:Bootv}.

Note also that, from the estimate on commutator~\eqref{est:parGamma}, the global Sobolev inequality~\eqref{est:SobolevGlobal} and the bootstrap assumption~\eqref{est:Bootw}, for all $t\in [0, T_{*}(\vec{w}_{0},\vec{v}_{0}))$, we have the following $L^{\infty}$ estimates related to the wave component $w$,
\begin{equation}\label{est:pointwavew}
\begin{aligned}
\sum_{|I|\le N-3}\langle t\rangle^{-\delta}\left\|\Gamma^{I}w(t,x)\right\|_{L_{x}^{\infty}}+\sum_{|I|\le N-9}\langle t\rangle^{-\delta}\|S\Gamma^{I}w(t,x)\|_{L_{x}^{\infty}}&\lesssim C_{0}\varepsilon \langle t\rangle^{-\frac{1}{2}},\\
\sum_{|I|\le N-4}\|\partial \Gamma^{I}w(t,x)\|_{L_{x}^{\infty}}+\sum_{|I|\le N-4}\langle t\rangle^{-\frac{1}{2}-2\delta}\|S\Gamma^{I}w(t,x)\|_{L_{x}^{\infty}}&\lesssim C_{0}\varepsilon \langle t\rangle^{-\frac{1}{2}}.
\end{aligned}
\end{equation}

From~\eqref{est:Bootv} and the H\"older inequality, for all $t\in [0,T_{*}(\vec{w}_{0},\vec{v}_{0}))$, we also have 
\begin{equation}\label{est:LtLxuGiu}
\begin{aligned}
&\int_{0}^{t} \langle s\rangle^{-\frac{1}{2}}\left(\left\|\frac{\Gamma^{I}v}{\langle s-r \rangle^{\frac{3}{4}}}\right\|_{L_{x}^{2}}+\left\|\frac{G_{i}\Gamma^{I}v}{\langle s-r\rangle^{\frac{3}{4}}}\right\|_{L_{x}^{2}}\right)\d s \\
&\lesssim \left(\int_{0}^{t} \langle s\rangle^{-1+\delta}\right)^{\frac{1}{2}}
\left(\int_{0}^{t}\langle s\rangle^{-\delta}\left(\left\|\frac{\Gamma^{I}v}{\langle s-r \rangle^{\frac{3}{4}}}\right\|^{2}_{L_{x}^{2}}+\left\|\frac{G_{i}\Gamma^{I}v}{\langle s-r\rangle^{\frac{3}{4}}}\right\|^{2}_{L_{x}^{2}}\right)\d s\right)^{\frac{1}{2}}\\
&\lesssim \left(\langle t\rangle^{\delta}\right)^{\frac{1}{2}}\left(C_{0}\varepsilon\langle t\rangle^{\delta}\right)^{\frac{1}{2}}\lesssim C^{\frac{1}{2}}_{0}\varepsilon^{\frac{1}{2}}\langle t\rangle^{\delta},\quad \mbox{for}\ I\in\mathbb{N}^{6}\ \mbox{with}\ |I|\le N\ \mbox{and}\ i=1,2.
\end{aligned}
\end{equation}

\subsection{Estimates on the nonlinear terms}
In this subsection, we establish the estimates on the nonlinear terms that will be needed in the proof of Proposition~\ref{prop:main}.

We start with the technical estimates related to $\Gamma^{I}Q$.
\begin{lemma}\label{le:estQ0}
For all $t\in [0,T_{*}(\vec{w}_{0},\vec{v}_{0}))$, the following estimates hold. 
	\begin{enumerate}
		\item \emph{Weight $L_{x}^{2}$ estimates on $\Gamma^{I}Q$.} We have 
		\begin{equation}\label{est:L2Q0}
		\sum_{|I|\le N-1}\left\|\langle t+r\rangle \Gamma^{I}Q(w,v)\right\|_{L_{x}^{2}}\lesssim C_{0}^{2}\varepsilon^{2}\langle t\rangle^{-\frac{1}{2}+2\delta}.
		\end{equation}
		
		\item \emph{$L_{t}^{1}L_{x}^{2}$ estimates on $\Gamma^{I}Q$.} We have  
	\begin{align}
	\sum_{|I|\le N-1}\int_{0}^{t}\left\|\Gamma^{I} Q(w,v)\right\|_{L^{2}_{x}}\d s &\lesssim C_{0}^{2}\varepsilon^{2},\label{est:LtLx1Q01}\\
	\sum_{|I|\le N}\int_{0}^{t}\left\|\Gamma^{I} Q(w,v)\right\|_{L^{2}_{x}}\d s &\lesssim C_{0}^{\frac{3}{2}}\varepsilon^{\frac{3}{2}}\langle t\rangle^{\delta}\label{est:LtLx1Q02}.
	\end{align}
	
	\item \emph{Weighted $L_{t}^{1}L_{x}^{2}$ estimates on $\Gamma^{I}Q$.} We have 
	\begin{align}
	\sum_{|I|\le N-6}\int_{0}^{t}\left\|\langle s+r\rangle\Gamma^{I}Q(w,v)\right\|_{L^{2}_{x}}\d s&\lesssim C_{0}^{2}\varepsilon^{2}\langle t\rangle^{\delta},\label{est:LtLxwQ01}\\
	\sum_{|I|\le N-1}\int_{0}^{t}\left\|\langle s+r\rangle\Gamma^{I}Q(w,v)\right\|_{L^{2}_{x}}\d s&\lesssim C_{0}^{2}\varepsilon^{2}\langle t \rangle^{\frac{1}{2}+2\delta}.\label{est:LtLxwQ02}
	\end{align}

\end{enumerate}
\end{lemma}
\begin{proof}
	Proof of (i). From~\eqref{est:txQ0} and $N\ge 14$, we see that
	\begin{equation*}
	\sum_{|I|\le N-1}\left\|\langle t+r\rangle \Gamma^{I}Q(w,v)\right\|_{L_{x}^{2}}\lesssim \mathcal{J}_{11}(t)+\mathcal{J}_{12}(t)+\mathcal{J}_{13}(t)+\mathcal{J}_{14}(t),
	\end{equation*}
	where
		\begin{equation*}
	\mathcal{J}_{11}=\sum_{|I_{1}|\le N-3 }\sum_{|I_{2}|\le N}\left\|\Gamma^{I_{1}}w\Gamma^{I_{2}}v\right\|_{L^{2}_{x}},\ 
	\mathcal{J}_{12}=\sum_{|I_{1}|\le N}\sum_{|I_{2}|\le N-5}\left\|\Gamma^{I_{1}}w\Gamma^{I_{2}}v\right\|_{L^{2}_{x}},
	\end{equation*}
	\begin{equation*}
	\mathcal{J}_{13}=\sum_{|I_{1}|\le N-9 }\sum_{|I_{2}|\le N}\left\|S\Gamma^{I_{1}}w\Gamma^{I_{2}}v\right\|_{L^{2}_{x}},\
	\mathcal{J}_{14}=\sum_{|I_{1}|\le N-1}\sum_{|I_{2}|\le N-5}\left\|S\Gamma^{I_{1}}w\Gamma^{I_{2}}v\right\|_{L^{2}_{x}}.
	\end{equation*}
	Using~\eqref{est:Bootw},~\eqref{est:Bootv} and~\eqref{est:pointwavew}, we have
	\begin{equation*}
	\mathcal{J}_{11}(t)\lesssim \sum_{|I_{1}|\le N-3 }\sum_{|I_{2}|\le N}\left\|\Gamma^{I_{1}}w\right\|_{L^{\infty}_{x}}\left\|\Gamma^{I_{2}}v\right\|_{L^{2}_{x}}
	\lesssim C_{0}^{2}\varepsilon^{2}t^{-\frac{1}{2}+2\delta},
	\end{equation*}
	\begin{equation*}
	\mathcal{J}_{12}(t)\lesssim \sum_{|I_{1}|\le N}\sum_{|I_{2}|\le N-5} \left\|\Gamma^{I_{1}}w\right\|_{L^{2}_{x}}\left\|\Gamma^{I_{2}}v\right\|_{L^{\infty}_{x}}
	\lesssim C_{0}^{2}\varepsilon^{2}t^{-1+\delta}, \
	\end{equation*}
	\begin{equation*}
	\mathcal{J}_{13}(t)\lesssim \sum_{|I_{1}|\le N-9 }\sum_{|I_{2}|\le N} \left\|S\Gamma^{I_{1}}w\right\|_{L_{x}^{\infty}}\left\|\Gamma^{I_{2}}v\right\|_{L^{2}_{x}}\lesssim
	C_{0}^{2}\varepsilon^{2}t^{-\frac{1}{2}+2\delta}, \
	\end{equation*}
	\begin{equation*}
\ \ \ \ 	\mathcal{J}_{14}(t)\lesssim \sum_{|I_{1}|\le N-1}\sum_{|I_{2}|\le N-5}\left\|S\Gamma^{I_{1}}w\right\|_{L_{x}^{2}}\left\|\Gamma^{I_{2}}v\right\|_{L^{\infty}_{x}}\lesssim
	C_{0}^{2}\varepsilon^{2}t^{-\frac{1}{2}+2\delta}. 
	\end{equation*}
	Combining the above estimates, we obtain~\eqref{est:L2Q0}.
	
	Proof of (ii). First, from~\eqref{est:L2Q0} and $0<\delta \ll 1$, we have 
	\begin{equation*}
	\sum_{|I|\le N-1}\left\|\Gamma^{I}Q(w,v)\right\|_{L^{2}_{x}}\lesssim \langle t\rangle^{-1}	\sum_{|I|\le N-1}\left\|\langle t+r\rangle \Gamma^{I}Q(w,v)\right\|_{L_{x}^{2}}\lesssim C_{0}^{2}\varepsilon^{2}\langle t\rangle^{-\frac{5}{4}},
	\end{equation*}
	which implies~\eqref{est:LtLx1Q01}. Second, from~\eqref{est:Q0decayG}, \eqref{equ:GammaQ0} and $N\ge 14$, we infer
	\begin{equation*}
	\sum_{|I|\le N}\int_{0}^{t}\left\|\Gamma^{I}Q(w,v)\right\|_{L_{x}^{2}}\d s\lesssim \mathcal{J}_{21}(t)+\mathcal{J}_{22}(t)+\mathcal{J}_{23}(t)+\mathcal{J}_{24}(t),
	\end{equation*}
	where
	\begin{align*}
	\mathcal{J}_{21}(t)&=\sum_{i=1,2}\sum_{\substack{|I_{1}|\le N-9\\ |I_{1}|+|I_{2}|\le N}}\int_{0}^{t}\left\|\left(G_{i}\Gamma^{I_{1}}w\right)\left(\partial\Gamma^{I_{2}} v\right)\right\|_{L_{x}^{2}}\d s,\\
	\mathcal{J}_{22}(t)&=\sum_{i=1,2}\sum_{\substack{|I_{2}|\le N-6\\ |I_{1}|+|I_{2}|\le N}}\int_{0}^{t}\left\|\left(G_{i}\Gamma^{I_{1}}w\right)\left(\partial\Gamma^{I_{2}} v\right)\right\|_{L_{x}^{2}}\d s,
	\end{align*}
	\begin{align*}
	\mathcal{J}_{23}(t)&=\sum_{i=1,2}\sum_{\substack{|I_{2}|\le N-6\\ |I_{1}|+|I_{2}|\le N}}\int_{0}^{t}\left\|\left(\partial\Gamma^{I_{1}} w\right)\left(G_{i}\Gamma^{I_{2}}v\right)\right\|_{L_{x}^{2}}\d s,\\
	\mathcal{J}_{24}(t)&=\sum_{i=1,2}\sum_{\substack{|I_{1}|\le N-9\\ |I_{1}|+|I_{2}|\le N}}\int_{0}^{t}\left\|\left(\partial\Gamma^{I_{1}} w\right)\left(G_{i}\Gamma^{I_{2}}v\right)\right\|_{L_{x}^{2}}\d s.
	\end{align*}
	Using~\eqref{est:decaypm},~\eqref{est:Bootv},~\eqref{est:pointwavew} and $0<\delta\ll1$, we have
	\begin{align*}
	\mathcal{J}_{21}
	&\lesssim \sum_{\substack{|I_{1}|\le N-9\\ |I_{1}|+|I_{2}|\le N}}\int_{0}^{t}\langle s\rangle^{-1}\left\|S\Gamma^{I_{1}}w\right\|_{L_{x}^{\infty}}\left\|\partial \Gamma^{I_{2}}v\right\|_{L_{x}^{2}}\d s\\
	&+\sum_{\substack{1\le |I_{1}|\le N-8\\ |I_{1}|+|I_{2}|\le N+1}}\int_{0}^{t}\langle s\rangle^{-1}\left\|\Gamma^{I_{1}}w\right\|_{L^{\infty}_{x}}\left\|\partial \Gamma^{I_{2}}v\right\|_{L_{x}^{2}}\d s\lesssim C_{0}^{2}\varepsilon^{2}\int_{0}^{t}\langle s\rangle^{-\frac{5}{4}}\d s\lesssim C_{0}^{2}\varepsilon^{2}.
	\end{align*}
	Then, using again~\eqref{est:Bootw},~\eqref{est:Bootv},~\eqref{est:pointwavew} and $|G_{i}|\lesssim |\partial|$, we infer
	\begin{align*}
	\mathcal{J}_{22}
	&\lesssim \sum_{\substack{1\le |I_{2}|\le N-5\\ |I_{1}|+|I_{2}|\le N+1}}
	\int_{0}^{t}\left\|\partial \Gamma^{I_{1}}w\right\|_{L^{2}_{x}}\left\|\Gamma^{I_{2}}v\right\|_{L_{x}^{\infty}}\d s\lesssim C_{0}^{2}\varepsilon^{2}\int_{0}^{t}\langle s \rangle^{-1+\delta}\d s \lesssim C_{0}^{2}\varepsilon^{2}\langle t\rangle^{\delta},\\
	\mathcal{J}_{23}
	&\lesssim \sum_{\substack{1\le |I_{2}|\le N-5\\ |I_{1}|+|I_{2}|\le N+1}}
	\int_{0}^{t}\left\|\partial \Gamma^{I_{1}}w\right\|_{L^{2}_{x}}\left\|\Gamma^{I_{2}}v\right\|_{L_{x}^{\infty}}\d s\lesssim C_{0}^{2}\varepsilon^{2}\int_{0}^{t}\langle s \rangle^{-1+\delta}\d s \lesssim C_{0}^{2}\varepsilon^{2}\langle t\rangle^{\delta}.
	\end{align*}
	Next, from~\eqref{est:Bootw} and~\eqref{est:LtLxuGiu}, we see that
	\begin{equation*}
	\begin{aligned}
	\mathcal{J}_{24}
	&\lesssim \sum_{i=1,2}\sum_{\substack{|I_{1}|\le N-9\\ |I_{1}|+|I_{2}|\le N}}\int_{0}^{t}\left\|\langle s-r\rangle^{\frac{3}{4}}\left(\partial\Gamma^{I_{1}} w\right)\right\|_{L_{x}^{\infty}}\left\|\frac{G_{i}\Gamma^{I_{2}}v}{\langle s-r \rangle^{\frac{3}{4}}}\right\|_{L_{x}^{2}}\d s\\
	&\lesssim C_{0}\varepsilon\sum_{i=1,2}\int_{0}^{t}\langle s\rangle^{-\frac{1}{2}}\left\|\frac{G_{i}\Gamma^{I_{2}}v}{\langle s-r \rangle^{\frac{3}{4}}}\right\|_{L^{2}_{x}}\d s\lesssim C^{\frac{3}{2}}_{0}\varepsilon^{\frac{3}{2}}\langle t\rangle^{\delta}.
	\end{aligned}
	\end{equation*}
	Combining the above estimates, we obtain~\eqref{est:LtLx1Q02}.
	
	Proof of (iii). First, from~\eqref{est:txQ0}, we have 
	\begin{equation*}
	\sum_{|I|\le N-6}\int_{0}^{t}\left\|\langle s+r\rangle \Gamma^{I}Q(w,v)\right\|_{L^{2}_{x}}\lesssim \mathcal{J}_{31}(t)+\mathcal{J}_{32}(t),
	\end{equation*}
	where
	\begin{align*}
	\mathcal{J}_{31}(t)&=\sum_{|I_{1}|\le N-5}\sum_{|I_{2}|\le N-5}\int_{0}^{t}\left\|\Gamma^{I_{1}}w\right\|_{L_{x}^{2}}\left\|\Gamma^{I_{2}}v\right\|_{L_{x}^{\infty}}\d s,\\
	\mathcal{J}_{32}(t)&=\sum_{|I_{1}|\le N-6}\sum_{|I_{2}|\le N-5}\int_{0}^{t}\left\|S\Gamma^{I_{1}}w\right\|_{L_{x}^{2}}\left\|\Gamma^{I_{2}}v\right\|_{L_{x}^{\infty}}\d s.
	\end{align*}
	Note that, using~\eqref{est:Bootw} and~\eqref{est:Bootv}, we obtain 
	\begin{equation*}
	\mathcal{J}_{31}(t)+\mathcal{J}_{32}(t)\lesssim C_{0}^{2}\varepsilon^{2}\int_{0}^{t}\langle s \rangle^{-1+\delta}\d s 
	\lesssim C_{0}^{2}\varepsilon^{2}\langle t\rangle^{\delta},
	\end{equation*}
	which implies~\eqref{est:LtLxwQ01}. Last, integrating~\eqref{est:L2Q0} on $[0,t]$, we obtain~\eqref{est:LtLxwQ02}.
	\end{proof}

Second, we introduce the $L_{t}^{1}L_{x}^{2}$ for the nonlinear terms that are related to hidden divergence type terms $F_{\alpha}$ and $H^{\alpha}$.
\begin{lemma}
	For all $t\in [0,T_{*}(\vec{w}_{0},\vec{v}_{0}))$, the following estimates hold.
	\begin{enumerate}
	\item \emph{$L_{t}^{1}L_{x}^{2}$ estimate for $\Gamma^{I}F_{\alpha}$ and $\Gamma^{I}H^{\alpha}$.} We have 
	\begin{equation}\label{est:LtLxFH}
	\sum_{|I|\le N}\int_{0}^{t}\left(\left\|\Gamma^{I}F_{\alpha}(w,v)\right\|_{L_{x}^{2}}+\left\|\Gamma^{I}H^{\alpha}(w,v)\right\|_{L_{x}^{2}}\right)\d s\lesssim C_{0}^{\frac{3}{2}}\varepsilon^{\frac{3}{2}}\langle t\rangle^{\delta}.
	\end{equation}
	
	\item \emph{$L_{t}^{1}L_{x}^{2}$ estimate for extra nonlinear terms.} We have 
	\begin{equation}\label{est:extraL2}
	\begin{aligned}
	\sum_{|I|\le N-5}\int_{0}^{t}\left(\left\|\Gamma^{I}Q(\partial w,v)\right\|_{L_{x}^{2}}+\left\|\Gamma^{I}(v^{2}\partial w)\right\|_{L_{x}^{2}}\right)\d s&\lesssim C_{0}^{2}\varepsilon^{2},\\
	\sum_{|I|\le N-5}\int_{0}^{t}\left(\left\|\Gamma^{I}\left(v\partial Q(w,v)\right)\right\|_{L_{x}^{2}}+\left\|\Gamma^{I}\left(Q(w,v)\partial w\right)\right\|_{L_{x}^{2}}\right)\d s&\lesssim C_{0}^{2}\varepsilon^{2}.
	\end{aligned}
	\end{equation}
\end{enumerate}
\end{lemma}
\begin{proof}
	Proof of (i).
	For all $\alpha=0,1,2$, we claim that
		\begin{equation}\label{est:LtLxvw}
	\sum_{|I|\le N}\int_{0}^{t}\left(\left\|\Gamma^{I}(v\partial_{\alpha}w)\right\|_{L_{x}^{2}}+\left\|\Gamma^{I}(v^{2}\partial_{\alpha}w)\right\|_{L_{x}^{2}}\right)\d s\lesssim C_{0}^{\frac{3}{2}}\varepsilon^{\frac{3}{2}}\langle t\rangle^{\delta}.
	\end{equation}
	Indeed, from $N\ge 14$, we decompose
	\begin{equation*}
\sum_{|I|\le N}\int_{0}^{t}\left(\left\|\Gamma^{I}(v\partial_{\alpha}w)\right\|_{L_{x}^{2}}+\left\|\Gamma^{I}(v^{2}\partial_{\alpha}w)\right\|_{L_{x}^{2}}\right)\d s\lesssim \mathcal{J}_{41}+\mathcal{J}_{42}+\mathcal{J}_{43}+\mathcal{J}_{44},
	\end{equation*}
	where
	\begin{equation*}
	\begin{aligned}
	\mathcal{J}_{41}&=\sum_{|I_{2}|\le N}\sum_{|I_{1}|\le N-5 }\int_{0}^{t}\left\|\left(\Gamma^{I_{1}}v\right)\left(\Gamma^{I_{2}}\partial_{\alpha}w\right)\right\|_{L_{x}^{2}}\d s,\\
	\mathcal{J}_{42}&=\sum_{|I_{1}|\le N}\sum_{|I_{2}|\le N-9}\int_{0}^{t}\left\|\left(\Gamma^{I_{1}}v\right)\left(\Gamma^{I_{2}}\partial_{\alpha}w\right)\right\|_{L_{x}^{2}}\d s, \quad \quad \quad \quad  \
	\end{aligned}
	\end{equation*}
		\begin{equation*}
	\begin{aligned}
	\mathcal{J}_{43}&=\sum_{\substack{|I_{1}|+|I_{2}|\le N-5\\ |I_{1}|+|I_{2}|+|I_{3}|\le N}}\int_{0}^{t}\left\|\left(\Gamma^{I_{1}}v\right)\left(\Gamma^{I_{2}}v\right)\left(\Gamma^{I_{3}}\partial_{\alpha}w\right)\right\|_{L_{x}^{2}}\d s,\\
	\mathcal{J}_{44}&=\sum_{\substack{|I_{1}|\le N-5,\ |I_{3}|\le N-5\\ |I_{1}|+|I_{2}|+|I_{3}|\le N}}\int_{0}^{t}\left\|\left(\Gamma^{I_{1}}v\right)\left(\Gamma^{I_{2}}v\right)\left(\Gamma^{I_{3}}\partial_{\alpha}w\right)\right\|_{L_{x}^{2}}\d s. 
	\end{aligned}
	\end{equation*}
	Then, using again~\eqref{est:parGamma},~\eqref{est:Bootw},~\eqref{est:Bootv},~\eqref{est:pointwavew} and~\eqref{est:LtLxuGiu}, we have 
	\begin{equation*}
	\mathcal{J}_{41}\lesssim \sum_{\substack{|I_{1}|\le N-5\\ |I_{1}|+|I_{2}|\le N}}\int_{0}^{t}\left\|\Gamma^{I_{1}}v\right\|_{L_{x}^{\infty}}\left\|\partial\Gamma^{I_{2}}w\right\|_{L_{x}^{2}}\d s\lesssim C_{0}^{2}\varepsilon^{2}\int_{0}^{t}\langle s\rangle^{-1+\delta}\d s\lesssim C_{0}^{2}\varepsilon^{2}\langle t\rangle^{\delta},
	\end{equation*}
	\begin{equation*}
	\begin{aligned}
	\mathcal{J}_{42}
	&\lesssim \sum_{\substack{|I_{2}|\le N-9\\ |I_{1}|+|I_{2}|\le N}}\int_{0}^{t}\left\|\frac{\Gamma^{I_{1}}v}{\langle s-r\rangle^{\frac{3}{4}}}\right\|_{L_{x}^{2}}\left\|\langle s-r\rangle^{\frac{3}{4}}\partial\Gamma^{I_{2}}w\right\|_{L_{x}^{\infty}}\d s \\
	&\lesssim C_{0}\varepsilon\int_{0}^{t}\langle s\rangle^{-\frac{1}{2}}\left\|\frac{\Gamma^{I_{1}}v}{\langle s-r\rangle^{\frac{3}{4}}}\right\|_{L_{x}^{2}}\d s \lesssim C_{0}^{\frac{3}{2}}\varepsilon^{\frac{3}{2}}\langle t\rangle^{\delta},
	\end{aligned}
	\end{equation*}
	\begin{equation*}
	\begin{aligned}
	\mathcal{J}_{43}
	&\lesssim\sum_{\substack{|I_{1}|+|I_{2}|\le N-5\\ |I_{1}|+|I_{2}|+|I_{3}|\le N}}\int_{0}^{t}\left\|\Gamma^{I_{1}}v\right\|_{L_{x}^{\infty}}\left\|\Gamma^{I_{2}}v\right\|_{L_{x}^{\infty}}\left\|\partial\Gamma^{I_{3}}w\right\|_{L_{x}^{2}}\d s\\
	&\lesssim C_{0}^{3}\varepsilon^{3}\int_{0}^{t}\langle s\rangle^{-1} \langle s\rangle^{-1} \langle s\rangle^{\delta}\d s \lesssim C_{0}^{3}\varepsilon^{3}\int_{0}^{t}\langle s\rangle^{-2+\delta}\d s 
	\lesssim C_{0}^{3}\varepsilon^{3},
	\end{aligned}
	\end{equation*}
	\begin{equation*}
	\begin{aligned}
	\mathcal{J}_{44}
	&\lesssim \sum_{\substack{|I_{1}|\le N-5,\ |I_{3}|\le N-5\\ |I_{1}|+|I_{2}|+|I_{3}|\le N}}\int_{0}^{t}\left\|\Gamma^{I_{1}}v\right\|_{L_{x}^{\infty}}\left\|\Gamma^{I_{2}}v\right\|_{L_{x}^{2}}\left\|\partial\Gamma^{I_{3}}w\right\|_{L_{x}^{\infty}}\d s\\
	&\lesssim C_{0}^{3}\varepsilon^{3}\int_{0}^{t}\langle s\rangle^{-1} \langle s\rangle^{\delta} \langle s\rangle^{-\frac{1}{2}}\d s \lesssim C_{0}^{3}\varepsilon^{3}\int_{0}^{t}\langle s\rangle^{-\frac{3}{2}+\delta}\d s 
	\lesssim C_{0}^{3}\varepsilon^{3}.
	\end{aligned}
	\end{equation*}
	Combining the above estimates, we obtain~\eqref{est:LtLxvw}. Last, we see that~\eqref{est:LtLxFH} follows from~\eqref{est:LtLxvw} and the definition of $F_{\alpha}$ and $H^{\alpha}$ in Lemma~\ref{le:equQ0Qab}.
	
	Proof of (ii). The proof is similar to~Lemma~\ref{le:estQ0}, and we omit it.
\end{proof}
Last, we introduce the following technical estimates related to quartic term $\Gamma^{I}G$.

\begin{lemma}For all $t\in [0,T_{*}(\vec{w}_{0},\vec{v}_{0}))$, the following estimates hold.
	\begin{enumerate}
		\item \emph{$L_{t}^{1}L_{x}^{1}$ estimate on $\Gamma^{I}G$.} We have 
		\begin{equation}\label{est:GBLtx1}
		\sum_{|I|\le N}\int_{0}^{t}\left\|\Gamma^{I}G(w,v)\right\|_{L^{1}_{x}}\d s\lesssim C_{0}^{4}\varepsilon^{4}.
		\end{equation}
		
		\item \emph{$L_{t}^{1}L_{x}^{2}$ estimate on $\Gamma^{I}G$.} We have  
			\begin{equation}\label{est:GBLtx2}
		\sum_{|I|\le N}\int_{0}^{t}\left\|\Gamma^{I}G(w,v)\right\|_{L^{2}_{x}}\d s\lesssim C_{0}^{4}\varepsilon^{4}.
		\end{equation}
		
		\item \emph{Weighted $L_{t}^{1}L_{x}^{1}$ estimate on $S\Gamma^{I} G$.} We have 
		\begin{equation}\label{est:GSBLtx1}
		\sum_{|I|\le N-8}\int_{0}^{t}(1+s)^{-\frac{1}{2}}\left\|S\Gamma^{I}G(w,v)\right\|_{L^{1}_{x}}\d s\lesssim C_{0}^{4}\varepsilon^{4}.
		\end{equation}
		
		\item \emph{Weighted $L_{t}^{1}L_{x}^{1}$ estimate on $S^{2}\Gamma^{I}G$.} We have 
		\begin{equation}\label{est:GSBLtx2}
	\sum_{|I|\le N-9}\int_{0}^{t}(1+s)^{-\frac{1}{2}}\left\|S^{2}\Gamma^{I}G(w,v)\right\|_{L^{1}_{x}}\d s\lesssim C_{0}^{4}\varepsilon^{4}.
	\end{equation}
	\end{enumerate} 
\end{lemma}

\begin{proof}
	Proof of (i). For all $Q\in \left\{Q_{0};\ Q_{\alpha\beta}, \alpha\ne \beta\right\}$, we claim that
	\begin{equation}\label{est:v2QLtLx1}
	\sum_{|I|\le N}\int_{0}^{t}\left\|\Gamma^{I}\left(v^{2}Q(w,v)\right)\right\|_{L_{x}^{1}}\d s \lesssim C_{0}^{4}\varepsilon^{4}.
	\end{equation}
	Indeed, from the definition of $Q$ and $N\ge 14$, we infer
	\begin{equation*}
	\sum_{|I|\le N}\int_{0}^{t}\left\|\Gamma^{I}\left(v^{2}Q(w,v)\right)\right\|_{L_{x}^{1}}\d s\lesssim \mathcal{J}_{51}+\mathcal{J}_{52},
	\end{equation*}
	where
	\begin{align*}
	\mathcal{J}_{51}&=\sum_{\substack{|I_{1}|\le N\\ |I_{2}|\le N-5}}\sum_{\substack{|I_{3}|\le N\\ |I_{4}|\le N-6}}\int_{0}^{t}\int_{\R^{2}}\left|\left(\Gamma^{I_{1}}v\right)\left(\Gamma^{I_{2}}v\right)\left(\partial\Gamma^{I_{3}}w\right)\left(\partial\Gamma^{I_{4}}v\right)\right|\d x\d s,\\
	\mathcal{J}_{52}&=\sum_{\substack{|I_{1}|\le N\\ |I_{2}|\le N-5}}\sum_{\substack{|I_{3}|\le N-5\\ |I_{4}|\le N}}\int_{0}^{t}\int_{\R^{2}}\left|\left(\Gamma^{I_{1}}v\right)\left(\Gamma^{I_{2}}v\right)\left(\partial\Gamma^{I_{3}}w\right)\left(\partial\Gamma^{I_{4}}v\right)\right|\d x\d s.
	\end{align*}
	From~\eqref{est:Bootw},~\eqref{est:Bootv},~\eqref{est:pointwavew} and the H\"older inequality, we have 
	\begin{align*}
	\mathcal{J}_{51}&\lesssim \sum_{\substack{|I_{1}|\le N\\ |I_{2}|\le N-5}}\sum_{\substack{|I_{3}|\le N\\ |I_{4}|\le N-6}}\int_{0}^{t}\left\|\Gamma^{I_{1}}v\right\|_{L_{x}^{2}}\left\|\Gamma^{I_{2}}v\right\|_{L_{x}^{\infty}}
	\left\|\partial\Gamma^{I_{3}}w\right\|_{L_{x}^{2}}\left\|\partial\Gamma^{I_{4}}v\right\|_{L_{x}^{\infty}}\d s\\
	&\lesssim C_{0}^{4}\varepsilon^{4}\int_{0}^{t}\langle s\rangle^{\delta}\langle s\rangle^{-1}\langle s\rangle^{\delta}\langle s\rangle^{-1}\d s \lesssim C_{0}^{4}\varepsilon^{4}\int_{0}^{t}\langle s\rangle^{-2+2\delta}\lesssim C_{0}^{4}\varepsilon^{4},
	\end{align*}
		\begin{align*}
	\mathcal{J}_{52}&\lesssim \sum_{\substack{|I_{1}|\le N\\ |I_{2}|\le N-5}}\sum_{\substack{|I_{3}|\le N-5\\ |I_{4}|\le N}}\int_{0}^{t}\left\|\Gamma^{I_{1}}v\right\|_{L_{x}^{2}}\left\|\Gamma^{I_{2}}v\right\|_{L_{x}^{\infty}}
	\left\|\partial\Gamma^{I_{3}}w\right\|_{L_{x}^{\infty}}\left\|\partial\Gamma^{I_{4}}v\right\|_{L_{x}^{2}}\d s\\
	&\lesssim C_{0}^{4}\varepsilon^{4}\int_{0}^{t}\langle s\rangle^{\delta}\langle s\rangle^{-1}\langle s\rangle^{-\frac{1}{2}}\langle s\rangle^{\delta}\d s \lesssim C_{0}^{4}\varepsilon^{4}\int_{0}^{t}\langle s\rangle^{-\frac{3}{2}+2\delta}\lesssim C_{0}^{4}\varepsilon^{4},
	\end{align*}
	which implies~\eqref{est:v2QLtLx1}. Last, we see that~\eqref{est:GBLtx1} follows from~\eqref{est:v2QLtLx1} and the definition of $G(w,v)$ in~Lemma~\ref{le:equQ0Qab}.
	
	Proof of (ii). For all $Q\in \left\{Q_{0};\ Q_{\alpha\beta}, \alpha\ne \beta\right\}$, we claim that
	\begin{equation}\label{est:v2QLtLx2}
	\sum_{|I|\le N}\int_{0}^{t}\left\|\Gamma^{I}\left(v^{2}Q(w,v)\right)\right\|_{L_{x}^{2}}\d s \lesssim C_{0}^{4}\varepsilon^{4}.
	\end{equation}
	Indeed, from the definition of $Q$, we infer
	\begin{equation*}
	\sum_{|I|\le N}\int_{0}^{t}\left\|\Gamma^{I}\left(v^{2}Q(w,v)\right)\right\|_{L_{x}^{1}}\d s\lesssim \mathcal{J}_{61}+\mathcal{J}_{62}+\mathcal{J}_{63}+\mathcal{J}_{64},
	\end{equation*}
	where
	\begin{equation*}
	\begin{aligned}
	\mathcal{J}_{61}&=\sum_{\substack{N-6 \le |I_{1}|\le N\\ 0\le |I_{2}|,|I_{3}|,|I_{4}|\le N-6}}
	\int_{0}^{t}\left\|\left(\Gamma^{I_{1}}v\right)\left(\Gamma^{I_{2}}v\right)\left(\partial\Gamma^{I_{3}}w\right)\left(\partial\Gamma^{I_{4}}v\right)\right\|_{L_{x}^{2}}\d s,\\
	\mathcal{J}_{62}&=\sum_{\substack{N-6 \le |I_{3}|\le N\\ 0\le |I_{1}|,|I_{2}|,|I_{4}|\le N-6}}
	\int_{0}^{t}\left\|\left(\Gamma^{I_{1}}v\right)\left(\Gamma^{I_{2}}v\right)\left(\partial\Gamma^{I_{3}}w\right)\left(\partial\Gamma^{I_{4}}v\right)\right\|_{L_{x}^{2}}\d s,\quad \quad 
	\end{aligned}
	\end{equation*}
	\begin{equation*}
	\begin{aligned}
	\mathcal{J}_{63}&=\sum_{\substack{N-6 \le |I_{4}|\le N\\ 0\le |I_{1}|,|I_{2}|,|I_{3}|\le N-6}}
	\int_{0}^{t}\left\|\left(\Gamma^{I_{1}}v\right)\left(\Gamma^{I_{2}}v\right)\left(\partial\Gamma^{I_{3}}w\right)\left(\partial\Gamma^{I_{4}}v\right)\right\|_{L_{x}^{2}}\d s,\\
		\mathcal{J}_{64}&=\sum_{ 0\le |I_{1}|,|I_{2}|,|I_{3}|,|I_{4}|\le N-6}
	\int_{0}^{t}\left\|\left(\Gamma^{I_{1}}v\right)\left(\Gamma^{I_{2}}v\right)\left(\partial\Gamma^{I_{3}}w\right)\left(\partial\Gamma^{I_{4}}v\right)\right\|_{L_{x}^{2}}\d s.
	\end{aligned}
	\end{equation*}
	From~\eqref{est:Bootw},~\eqref{est:Bootv} and~\eqref{est:pointwavew}, we see that
	\begin{equation*}
	\begin{aligned}
		\mathcal{J}_{61}
		&=\sum_{\substack{N-6 \le |I_{1}|\le N\\ 0\le |I_{2}|,|I_{3}|,|I_{4}|\le N-6}}
	\int_{0}^{t}\left\|\Gamma^{I_{1}}v\right\|_{L_{x}^{2}}\left\|\Gamma^{I_{2}}v\right\|_{L_{x}^{\infty}}\left\|\partial\Gamma^{I_{3}}w\right\|_{L_{x}^{\infty}}\left\|\partial\Gamma^{I_{4}}v\right\|_{L_{x}^{\infty}}\d s\\
	&\lesssim C_{0}^{4}\varepsilon^{4}\int_{0}^{t}\langle s\rangle^{\delta} \langle s\rangle^{-1} \langle s\rangle^{-\frac{1}{2}} \langle s\rangle^{-1} \d s \lesssim C_{0}^{4}\varepsilon^{4}\int_{0}^{t}\langle s\rangle^{-\frac{3}{2}}\d s\lesssim C_{0}^{4}\varepsilon^{4},
	\end{aligned}
	\end{equation*}
	\begin{equation*}
		\begin{aligned}
	\mathcal{J}_{62}
	&=\sum_{\substack{N-6 \le |I_{3}|\le N\\ 0\le |I_{1}|,|I_{2}|,|I_{4}|\le N-6}}
	\int_{0}^{t}\left\|\Gamma^{I_{1}}v\right\|_{L_{x}^{\infty}}\left\|\Gamma^{I_{2}}v\right\|_{L_{x}^{\infty}}\left\|\partial\Gamma^{I_{3}}w\right\|_{L_{x}^{2}}\left\|\partial\Gamma^{I_{4}}v\right\|_{L_{x}^{\infty}}\d s\\
	&\lesssim C_{0}^{4}\varepsilon^{4}\int_{0}^{t}\langle s\rangle^{-1} \langle s\rangle^{-1} \langle s\rangle^{\delta} \langle s\rangle^{-1} \d s \lesssim C_{0}^{4}\varepsilon^{4}\int_{0}^{t}\langle s\rangle^{-2}\d s\lesssim C_{0}^{4}\varepsilon^{4},
	\end{aligned}
	\end{equation*}
	\begin{equation*}
		\begin{aligned}
	\mathcal{J}_{63}
	&=\sum_{\substack{N-6 \le |I_{4}|\le N\\ 0\le |I_{1}|,|I_{2}|,|I_{3}|\le N-6}}
	\int_{0}^{t}\left\|\Gamma^{I_{1}}v\right\|_{L_{x}^{\infty}}\left\|\Gamma^{I_{2}}v\right\|_{L_{x}^{\infty}}\left\|\partial\Gamma^{I_{3}}w\right\|_{L_{x}^{\infty}}\left\|\partial\Gamma^{I_{4}}v\right\|_{L_{x}^{2}}\d s\\
	&\lesssim C_{0}^{4}\varepsilon^{4}\int_{0}^{t}\langle s\rangle^{-1} \langle s\rangle^{-1} \langle s\rangle^{-\frac{1}{2}} \langle s\rangle^{\delta} \d s \lesssim C_{0}^{4}\varepsilon^{4}\int_{0}^{t}\langle s\rangle^{-\frac{3}{2}}\d s\lesssim C_{0}^{4}\varepsilon^{4},
	\end{aligned}
	\end{equation*}
	\begin{equation*}
	\begin{aligned}
	\mathcal{J}_{64}
	&=\sum_{\substack{0\le |I_{1}||I_{2}|,|I_{3}|,|I_{4}|\le N-6}}
	\int_{0}^{t}\left\|\Gamma^{I_{1}}v\right\|_{L_{x}^{\infty}}\left\|\Gamma^{I_{2}}v\right\|_{L_{x}^{\infty}}\left\|\partial\Gamma^{I_{3}}w\right\|_{L_{x}^{\infty}}\left\|\partial\Gamma^{I_{4}}v\right\|_{L_{x}^{2}}\d s\\
	&\lesssim C_{0}^{4}\varepsilon^{4}\int_{0}^{t}\langle s\rangle^{-1} \langle s\rangle^{-1} \langle s\rangle^{-\frac{1}{2}} \langle s\rangle^{\delta} \d s \lesssim C_{0}^{4}\varepsilon^{4}\int_{0}^{t}\langle s\rangle^{-\frac{3}{2}}\d s\lesssim C_{0}^{4}\varepsilon^{4}.
	\end{aligned}
	\end{equation*}
	Combining the above estimates, we obtain~\eqref{est:v2QLtLx2}. Last, we see that~\eqref{est:GBLtx2} follows from~\eqref{est:v2QLtLx2} and the definition of $G(w,v)$ in Lemma~\ref{le:equQ0Qab}.
	
	Proof of (iii). For all $Q\in \left\{Q_{0};\ Q_{\alpha\beta},\ \alpha\ne \beta\right\}$, we claim that
	\begin{equation}\label{est:v2SQLtLx1}
	\sum_{|I|\le N-8}\left\|S\Gamma^{I}\left(v^{2}Q(w,v)\right)\right\|_{L_{x}^{2}}\lesssim C_{0}^{4}\varepsilon^{4}\langle s\rangle^{-1},\quad \mbox{for all}\ s\in [0,T_{*}(\vec{w}_{0},\vec{v}_{0})).
	\end{equation}
	Indeed, from the inequality $|S|\lesssim \langle s+r\rangle |\partial|$, we infer
	\begin{equation*}
	\sum_{|I|\le N-8}\left\|S\Gamma^{I}\left(v^{2}Q(w,v)\right)\right\|_{L_{x}^{1}}\lesssim
	\sum_{|I|\le N-7}\left\|\langle s+r\rangle\Gamma^{I}\left(v^{2}Q(w,v)\right)\right\|_{L_{x}^{1}}.
	\end{equation*}
	Then, using again~\eqref{est:txQ0}, we have 
	\begin{equation*}
	\sum_{|I|\le N-7}\left\|\langle s+r\rangle\Gamma^{I}\left(v^{2}Q(w,v)\right)\right\|_{L_{x}^{1}}\lesssim \mathcal{J}_{71}+\mathcal{J}_{72},
	\end{equation*}
	where
	\begin{equation*}
	\begin{aligned}
	\mathcal{J}_{71}&=\sum_{ \substack{0\le |I_{1}|,|I_{2}|\le N-6\\ 0\le |I_{3}|,|I_{4}|\le N-6 }}\left\|\left(\Gamma^{I_{1}}v\right)\left(\Gamma^{I_{2}}v\right)\left(\Gamma^{I_{3}}w\right)\left(\Gamma^{I_{4}}v\right)\right\|_{L^{1}_{x}},\\
	\mathcal{J}_{72}&=\sum_{ \substack{0\le |I_{1}|,|I_{2}|\le N-6\\ 0\le |I_{3}|,|I_{4}|\le N-6 }}\left\|\left(\Gamma^{I_{1}}v\right)\left(\Gamma^{I_{2}}v\right)\left(S\Gamma^{I_{3}}w\right)\left(\Gamma^{I_{4}}v\right)\right\|_{L^{1}_{x}}.
	\end{aligned}
	\end{equation*}
	Note that, from~\eqref{est:Bootw},~\eqref{est:Bootv} and the H\"older inequality, we have 
	\begin{equation*}
	\begin{aligned}
	\mathcal{J}_{71}&\lesssim \sum_{ \substack{0\le |I_{1}|,|I_{2}|\le N-6\\ 0\le |I_{3}|,|I_{4}|\le N-6 }}\left\|\Gamma^{I_{1}}v\right\|_{L_{x}^{\infty}}\left\|\Gamma^{I_{2}}v\right\|_{L_{x}^{\infty}}\left\|\Gamma^{I_{3}}w\right\|_{L_{x}^{2}}\left\|\Gamma^{I_{4}}v\right\|_{L_{x}^{2}}\\
	&\lesssim C_{0}^{4}\varepsilon^{4}\langle s\rangle^{-1} \langle s\rangle^{-1}\langle s \rangle^{\delta}\langle s\rangle^{\delta} \lesssim C_{0}^{4}\varepsilon^{4}\langle s\rangle^{-2+2\delta}\lesssim C_{0}^{4}\varepsilon^{4}\langle s\rangle^{-1},
	\end{aligned}
	\end{equation*}
	\begin{equation*}
	\begin{aligned}
	\mathcal{J}_{72}&\lesssim \sum_{ \substack{0\le |I_{1}|,|I_{2}|\le N-6\\ 0\le |I_{3}|,|I_{4}|\le N-6 }}\left\|\Gamma^{I_{1}}v\right\|_{L_{x}^{\infty}}\left\|\Gamma^{I_{2}}v\right\|_{L_{x}^{\infty}}\left\|S\Gamma^{I_{3}}w\right\|_{L_{x}^{2}}\left\|\Gamma^{I_{4}}v\right\|_{L_{x}^{2}}\\
	&\lesssim C_{0}^{4}\varepsilon^{4}\langle s\rangle^{-1} \langle s\rangle^{-1}\langle s \rangle^{\delta}\langle s\rangle^{\delta} \lesssim C_{0}^{4}\varepsilon^{4}\langle s\rangle^{-2+2\delta}\lesssim C_{0}^{4}\varepsilon^{4}\langle s\rangle^{-1},
	\end{aligned}
	\end{equation*}
	which implies~\eqref{est:v2SQLtLx1}. Last, multiplying both sides of~\eqref{est:v2SQLtLx1} by $(1+s)^{-\frac{1}{2}}$ and then integrating on $[0,t]$, we obtain~\eqref{est:GSBLtx1}.
	\end{proof}

Proof of (iv). For all $Q\in \left\{Q_{0};\ Q_{\alpha\beta},\alpha\ne \beta\right\}$, we claim that
\begin{equation}\label{est:v2SQLtLx2}
\sum_{|I|\le N-9}\left\|S^{2}\Gamma^{I}\left(v^{2}Q(w,v)\right)\right\|_{L_{x}^{1}}\lesssim C_{0}^{4}\varepsilon^{4}\langle s\rangle^{-1+2\delta},\quad \mbox{for all}\ s\in [0,T_{*}).
\end{equation}
Indeed, from the inequality $|S^{2}|\lesssim \langle s+r\rangle^{2} |\partial\partial|$, we infer
\begin{equation*}
\sum_{|I|\le N-9}\left\|S^{2}\Gamma^{I}\left(v^{2}Q(w,v)\right)\right\|_{L_{x}^{1}}\lesssim
\sum_{|I|\le N-7}\left\|\langle s+r\rangle^{2}\Gamma^{I}\left(v^{2}Q(w,v)\right)\right\|_{L_{x}^{1}}.
\end{equation*}
Then, using again~\eqref{est:txQ0}, we have 
\begin{equation*}
\sum_{|I|\le N-7}\left\|\langle s+r\rangle^{2}\Gamma^{I}\left(v^{2}Q(w,v)\right)\right\|_{L_{x}^{1}}\lesssim \mathcal{J}_{81}+\mathcal{J}_{82},
\end{equation*}
where
\begin{equation*}
\begin{aligned}
\mathcal{J}_{81}&=\sum_{ \substack{0\le |I_{1}|,|I_{2}|\le N-6\\ 0\le |I_{3}|,|I_{4}|\le N-6 }}\left\|\langle s+r \rangle\left(\Gamma^{I_{1}}v\right)\left(\Gamma^{I_{2}}v\right)\left(\Gamma^{I_{3}}w\right)\left(\Gamma^{I_{4}}v\right)\right\|_{L^{1}_{x}},\\
\mathcal{J}_{82}&=\sum_{ \substack{0\le |I_{1}|,|I_{2}|\le N-6\\ 0\le |I_{3}|,|I_{4}|\le N-6 }}\left\|\langle s+r \rangle\left(\Gamma^{I_{1}}v\right)\left(\Gamma^{I_{2}}v\right)\left(S\Gamma^{I_{3}}w\right)\left(\Gamma^{I_{4}}v\right)\right\|_{L^{1}_{x}}.
\end{aligned}
\end{equation*}
Note that, from~\eqref{est:Bootw},~\eqref{est:Bootv} and the H\"older inequality, we have 
\begin{equation*}
\begin{aligned}
\mathcal{J}_{81}&\lesssim \sum_{ \substack{0\le |I_{1}|,|I_{2}|\le N-6\\ 0\le |I_{3}|,|I_{4}|\le N-6 }}\left\|\langle s+r \rangle\Gamma^{I_{1}}v\right\|_{L_{x}^{\infty}}\left\|\Gamma^{I_{2}}v\right\|_{L_{x}^{\infty}}\left\|\Gamma^{I_{3}}w\right\|_{L_{x}^{2}}\left\|\Gamma^{I_{4}}v\right\|_{L_{x}^{2}}\\
&\lesssim C_{0}^{4}\varepsilon^{4} \langle s\rangle^{-1}\langle s \rangle^{\delta}\langle s\rangle^{\delta} \lesssim C_{0}^{4}\varepsilon^{4}\langle s\rangle^{-1+2\delta},
\end{aligned}
\end{equation*}
\begin{equation*}
\begin{aligned}
\mathcal{J}_{82}&\lesssim \sum_{ \substack{0\le |I_{1}|,|I_{2}|\le N-6\\ 0\le |I_{3}|,|I_{4}|\le N-6 }}\left\|\langle s+r \rangle\Gamma^{I_{1}}v\right\|_{L_{x}^{\infty}}\left\|\Gamma^{I_{2}}v\right\|_{L_{x}^{\infty}}\left\|S\Gamma^{I_{3}}w\right\|_{L_{x}^{2}}\left\|\Gamma^{I_{4}}v\right\|_{L_{x}^{2}}\\
&\lesssim C_{0}^{4}\varepsilon^{4} \langle s\rangle^{-1}\langle s \rangle^{\delta}\langle s\rangle^{\delta} \lesssim C_{0}^{4}\varepsilon^{4}\langle s\rangle^{-1+2\delta},
\end{aligned}
\end{equation*}
which implies~\eqref{est:v2SQLtLx2}.  Last, multiplying both sides of~\eqref{est:v2SQLtLx2} by $(1+s)^{-\frac{1}{2}}$ and then integrating on $[0,t]$, we obtain~\eqref{est:GSBLtx2}.
\subsection{End of the proof of Proposition~\ref{prop:main}} We are in a position to complete the proof of Proposition~\ref{prop:main}.

\begin{proof}[Proof of Proposition~\ref{prop:main}]\label{Se:mainpro}
	For all $(\vec{w}_{0},\vec{v}_{0})$ satisfying the smallness condition~\eqref{est:small}, we consider the corresponding solution $(w,v)$ of~\eqref{equ:waveKG}. From the definition of initial data, we have 
	\begin{equation}\label{est:inita}
	\sum_{|I|\le N}\left(\mathcal{E}(0,\Gamma^{I}w)^{\frac{1}{2}}+\mathcal{G}(0,\Gamma^{I}w)^{\frac{1}{2}}+\mathcal{E}_{1}(0,\Gamma^{I}v)^{\frac{1}{2}}\right)\lesssim \varepsilon.
	\end{equation}
	
	{\textbf{Step 1.}} Closing the estimates in $\mathcal{E}(t,\Gamma^{I}w)$ and $\mathcal{E}_{1}(t,\Gamma^{I}v)$. 
	Using~\eqref{equ:waveKG}, \eqref{est:EnerE}, \eqref{est:EnergK}, ~\eqref{est:LtLx1Q01},~\eqref{est:LtLx1Q02} and~\eqref{est:inita}, we deduce that
	\begin{equation*}
	\begin{aligned}
	&\sum_{|I|\le N-1}\left(\mathcal{E}(t,\Gamma^{I}w)^{\frac{1}{2}}+\mathcal{E}_{1}(t,\Gamma^{I}v)^{\frac{1}{2}}\right)\\
	&\lesssim 	\sum_{|I|\le N-1}\left(\mathcal{E}(0,\Gamma^{I}w)^{\frac{1}{2}}+\mathcal{E}_{1}(0,\Gamma^{I}v)^{\frac{1}{2}}\right)
	+\sum_{|I|\le N-1}\int_{0}^{t}\|\Gamma^{I}Q_{0}(w,v)\|_{L_{x}^{2}}\d s \\
	&+\sum_{\alpha\ne \beta}\sum_{|I|\le N-1}\int_{0}^{t}\|\Gamma^{I}Q_{\alpha\beta}(w,v)\|_{L_{x}^{2}}\d s 
	\lesssim \varepsilon+C_{0}^{2}\varepsilon^{2}\quad \mbox{on}\ [0,T_{*}(\vec{w}_{0},\vec{v}_{0})),
	\end{aligned}
	\end{equation*} 
	and
		\begin{equation*}
	\begin{aligned}
	&\sum_{|I|\le N}\left(\mathcal{E}(t,\Gamma^{I}w)^{\frac{1}{2}}+\mathcal{E}_{1}(t,\Gamma^{I}v)^{\frac{1}{2}}\right)\\
	&\lesssim 	\sum_{|I|\le N}\left(\mathcal{E}(0,\Gamma^{I}w)^{\frac{1}{2}}+\mathcal{E}_{1}(0,\Gamma^{I}v)^{\frac{1}{2}}\right)
	+\sum_{|I|\le N}\int_{0}^{t}\|\Gamma^{I}Q_{0}(w,v)\|_{L_{x}^{2}}\d s \\
	&+\sum_{\alpha\ne \beta}\sum_{|I|\le N}\int_{0}^{t}\|\Gamma^{I}Q_{\alpha\beta}(w,v)\|_{L_{x}^{2}}\d s 
	\lesssim \varepsilon+C_{0}^{2}\varepsilon^{2}\langle t\rangle^{\delta}\quad \mbox{on}\ [0,T_{*}(\vec{w}_{0},\vec{v}_{0})).
	\end{aligned}
	\end{equation*} 
	These strictly improve the bootstrap estimates of $\mathcal{E}(t,\Gamma^{I}w)$ and $\mathcal{E}_{1}(t,\Gamma^{I}v)$ in~\eqref{est:Bootw} and~\eqref{est:Bootv} for $C_{0}$ large enough and $\varepsilon$ small enough (depending on $C_{0}$).
	
	\smallskip
	{\textbf{Step 2.}} Closing the estimates in $\left\|\Gamma^{I}w\right\|_{L_{x}^{2}}$. We decompose the wave component $w$ as (see also \cite{Kataya})
	\begin{equation*}
	w=\Upsilon_{0}+\Upsilon_{1}+\partial^{\alpha}\Psi_{\alpha}+\partial_{\alpha}\Phi^{\alpha},
	\end{equation*}
	where $\Upsilon_{0}$, $\Upsilon_{1}$, $\Psi_{\alpha}$ and $\Phi^{\alpha}$ are the solutions for the following 2D linear homogeneous or inhomogeneous wave equations,
	\begin{equation}\label{equ:wavedecom}
	\left\{\begin{aligned}
	-\Box \Upsilon_{0}&=0,\quad \quad \quad \ \ \mbox{with}\ \ (\Upsilon_{0},\partial_{t}\Upsilon_{0})_{|t=0}=(w_{0},w_{1}),\\
	-\Box \Upsilon_{1}&=G(w,v),\quad \mbox{with}\ \ (\Upsilon_{1},\partial_{t}\Upsilon_{1})_{|t=0}=(0,0),\\
	-\Box \Psi_{\alpha}&=F_{\alpha}(w,v),\ \ \mbox{with}\ \ (\Psi_{\alpha},\partial_{t}\Psi_{\alpha})_{|t=0}=(0,0),\\
		-\Box \Phi^{\alpha}&=H^{\alpha}(w,v),\ \ \mbox{with}\ \ (\Phi^{\alpha},\partial_{t}\Phi^{\alpha})_{|t=0}=(0,0).
	\end{aligned}\right.
	\end{equation}
	First, from the smallness condition~\eqref{est:small}, 
	\begin{equation*}
	\sum_{|I|\le N}\left(\left\|\Gamma^{I}\Upsilon_{0}(0,x)\right\|_{L_{x}^{2}}+\left\|\partial_{t}\Gamma^{I}\Upsilon_{0}(0,x)\right\|_{L_{x}^{2}}+\left\|\partial_{t}\Gamma^{I}\Upsilon_{0}(0,x)\right\|_{L_{x}^{1}}\right)\lesssim \varepsilon.
	\end{equation*}
	Thus, using~\eqref{est:L2wave}, for all $I\in \mathbb{N}^{6}$ with $|I|\le N$, we see that
	\begin{equation}\label{est:L2Ups}
	\begin{aligned}
	\left\|\Gamma^{I}\Upsilon_{0}(t,x)\right\|_{L_{x}^{2}}
	&\lesssim \left\|\Gamma^{I}\Upsilon_{0}(0,x)\right\|_{L_{x}^{2}}+\log^{\frac{1}{2}}(2+t)\left\|\partial_{t}\Gamma^{I}\Upsilon_{0}(0,x)\right\|_{L_{x}^{2}}\\
	&+\log^{\frac{1}{2}}(2+t)\left\|\partial_{t}\Gamma^{I}\Upsilon_{0}(0,x)\right\|_{L_{x}^{1}}\lesssim \varepsilon\log^{\frac{1}{2}}(2+t).
	\end{aligned}
	\end{equation}
	Second, from~\eqref{est:L2waveinh},~\eqref{est:GBLtx1} and~\eqref{est:GBLtx2}, for all $I\in \mathbb{N}^{6}$ with $|I|\le N$, we see that
		\begin{equation}\label{est:L2UPP}
	\begin{aligned}
\left\|\Gamma^{I}\Upsilon_{1}\right\|_{L_{x}^{2}}
&\lesssim \log^{\frac{1}{2}} (2+t)\int_{0}^{t}\left\|\Gamma^{I}G(w,v)\right\|_{L^{2}_{x}}\d s\\
&+\log^{\frac{1}{2}}(2+t)\int_{0}^{t}\left\|\Gamma^{I}G(w,v)\right\|_{L^{1}_{x}}\d s\lesssim  C_{0}^{4}\varepsilon^{4}\log^{\frac{1}{2}}(2+t).
	\end{aligned}
	\end{equation}
	
	Then, from~\eqref{est:EnerE},~\eqref{est:LtLxFH}, for all $I\in \mathbb{N}^{6}$ with $|I|\le N$, we infer
	\begin{equation*}
	\begin{aligned}
	\sum_{\alpha=0}^{2}\mathcal{E}(t,\Gamma^{I}\Psi_{\alpha})^{\frac{1}{2}}&\lesssim \int_{0}^{t}\left\|\Gamma^{I}F_{\alpha}(w,v)\right\|_{L_{x}^{2}}\d s\lesssim C_{0}^{\frac{3}{2}}\varepsilon^{\frac{3}{2}}\langle t\rangle^{\delta},\\
	\sum_{\alpha=0}^{2}\mathcal{E}(t,\Gamma^{I}\Phi^{\alpha})^{\frac{1}{2}}&\lesssim \int_{0}^{t}\left\|\Gamma^{I}H^{\alpha}(w,v)\right\|_{L_{x}^{2}}\d s\lesssim C_{0}^{\frac{3}{2}}\varepsilon^{\frac{3}{2}}\langle t\rangle^{\delta}.
	\end{aligned}
	\end{equation*}
	Based on the above estimates, for all $I\in \mathbb{N}^{6}$ with $|I|\le N$, we see that  
	\begin{equation}\label{est:L2PsiPhi}
	\begin{aligned}
	\sum_{\alpha=0}^{2}\left\|\partial \Gamma^{I}\Psi_{\alpha}\right\|_{L_{x}^{2}}
	&\lesssim \sum_{\alpha=0}^{2}\mathcal{E}_{1}(t,\Gamma^{I}\Psi_{\alpha})^{\frac{1}{2}}\lesssim C_{0}^{\frac{3}{2}}\varepsilon^{\frac{3}{2}}\langle t\rangle^{\delta},\\
	\sum_{\alpha=0}^{2}\left\|\partial \Gamma^{I}\Phi^{\alpha}\right\|_{L_{x}^{2}}
	&\lesssim \sum_{\alpha=0}^{2}\mathcal{E}_{1}(t,\Gamma^{I}\Phi^{\alpha})^{\frac{1}{2}}\lesssim C_{0}^{\frac{3}{2}}\varepsilon^{\frac{3}{2}}\langle t\rangle^{\delta}.
	\end{aligned}
	\end{equation}
	Combining~\eqref{est:parGamma},~\eqref{est:L2Ups},~\eqref{est:L2UPP} and~\eqref{est:L2PsiPhi}, we obtain
	\begin{equation}\label{est:L2wimp}
	\begin{aligned}
	\sum_{|I|\le N}\|\Gamma^{I}w\|_{L_{x}^{2}}
	&\lesssim \sum_{\alpha=0}^{2}\sum_{k=0,1}\sum_{|I|\le N}\left(\left\|\Gamma^{I}\Upsilon_{k}\right\|_{L_{x}^{2}}+\left\|\partial\Gamma^{I}\Psi_{\alpha}\right\|_{L_{x}^{2}}+\left\|\partial\Gamma^{I}\Phi^{\alpha}\right\|_{L_{x}^{2}}\right)\\
	&\lesssim \left(\varepsilon+C_{0}^{4}\varepsilon^{4}\right)\log^{\frac{1}{2}}(2+t)+C_{0}^{\frac{3}{2}}\varepsilon^{\frac{3}{2}}\langle t\rangle^{\delta}\lesssim \left(\varepsilon+C_{0}^{\frac{3}{2}}\varepsilon^{\frac{3}{2}}\right)\langle t\rangle^{\delta}.
	\end{aligned}
	\end{equation}
	This strictly improves the bootstrap estimates of $\|\Gamma^{I}w\|_{L_{x}^{2}}$ in~\eqref{est:Bootw} for $C_{0}$ large enough and $\varepsilon$ small enough (depending on $C_{0}$).
	
	\smallskip
	{\textbf{Step 3.}} Closing the estimates in $\left\|S {\Gamma}^{I}w\right\|_{L_{x}^{2}}$. From~\eqref{est:ConE},~\eqref{est:LtLxwQ01},~\eqref{est:LtLxwQ02} and~\eqref{est:inita}, we see that
	\begin{equation*}
	\begin{aligned}
	\sum_{|I|\le N-6}\mathcal{G}(t,\Gamma^{I}w)^{\frac{1}{2}}
	&\lesssim \sum_{|I|\le N-6}\mathcal{G}(0,\Gamma^{I}w)^{\frac{1}{2}}+\sum_{|I|\le N-6}\int_{0}^{t}\|\langle s+r\rangle \Gamma^{I}Q_{0}(w,v)\|_{L_{x}^{2}}\d s\\
	&+\sum_{|I|\le N-6}\sum_{\alpha\ne \beta}\int_{0}^{t}\|\langle s+r\rangle \Gamma^{I}Q_{\alpha\beta}(w,v)\|_{L_{x}^{2}}\d s\lesssim \left(\varepsilon+C_{0}^{2}\varepsilon^{2}\right)\langle t\rangle^{{\delta}},
	\end{aligned}
	\end{equation*}
		\begin{equation*}
	\begin{aligned}
	\sum_{|I|\le N-1}\mathcal{G}(t,\Gamma^{I}w)
	&\lesssim \sum_{|I|\le N-1}\mathcal{G}(0,\Gamma^{I}w)+\sum_{|I|\le N-1}\int_{0}^{t}\|\langle s+r\rangle \Gamma^{I}Q_{0}(w,v)\|_{L_{x}^{2}}\d s\\
	&+\sum_{|I|\le N-1}\sum_{\alpha\ne\beta}\int_{0}^{t}\|\langle s+r\rangle \Gamma^{I}Q_{\alpha\beta}(w,v)\|_{L_{x}^{2}}\d s\lesssim \left(\varepsilon+C_{0}^{2}\varepsilon^{2}\right)\langle t\rangle^{\frac{1}{2}+2{\delta}}.
	\end{aligned}
	\end{equation*}
	Combining the above estimates with~\eqref{est:L2wimp}, we obtain
	\begin{equation*}
	\begin{aligned}
	\sum_{|I|\le N-6}\|S\Gamma^{I}w\|_{L_{x}^{2}}
	&\lesssim \sum_{|I|\le N-6}\mathcal{G}(t,\Gamma^{I}w)^{\frac{1}{2}}+\sum_{|I|\le N-6}\|\Gamma^{I}w\|_{L_{x}^{2}}\lesssim \left(\varepsilon+C_{0}^{\frac{3}{2}}\varepsilon^{\frac{3}{2}}\right)\langle t\rangle^{{\delta}},
	\end{aligned}
	\end{equation*}
		\begin{equation*}
	\sum_{|I|\le N-1}\|S\Gamma^{I}w\|_{L_{x}^{2}}
	\lesssim \sum_{|I|\le N-1}\mathcal{G}(t,\Gamma^{I}w)^{\frac{1}{2}}+\sum_{|I|\le N-1}\|\Gamma^{I}w\|_{L_{x}^{2}}\lesssim \left(\varepsilon+C_{0}^{\frac{3}{2}}\varepsilon^{\frac{3}{2}}\right)\langle t\rangle^{{\frac{1}{2}+2\delta}}.
	\end{equation*}
	These strictly improve the bootstrap estimates of $\|S\Gamma^{I}w\|_{L_{x}^{2}}$ in~\eqref{est:Bootw} for $C_{0}$ large enough and $\varepsilon$ small enough (depending on $C_{0}$).
	
	\smallskip
	{\textbf{Step 4.}} Closing the estimate in $\left|\partial \Gamma^{I} w\right|$. We split the proof into two parts according to different spacetime regions.
	
	Case I: Let $(t,x)\in \left\{(t,x)\in \R^{1+2}:r\ge 2t\right\}$. From the Step 2 and Step 3, we have 
	\begin{equation*}
	\sum_{|I|\le N-6}\left(\left\|S\Gamma^{I}w\right\|_{L_{x}^{2}}+\|\Gamma^{I}w\|_{L^{2}_{x}}\right)\lesssim \left(\varepsilon+C_{0}^{\frac{3}{2}}\varepsilon^{\frac{3}{2}}\right)\langle t\rangle^{\delta}.
	\end{equation*}
	Therefore, from~\eqref{est:parGamma} and~\eqref{est:SobolevGlobal}, we see that 
	\begin{equation*}
	\sum_{|I|\le N-9}\left(\left\|S\Gamma^{I}w\right\|_{L_{x}^{\infty}}+\left\|\Gamma^{I}w\right\|_{L_{x}^{\infty}}\right)\lesssim \left(\varepsilon+C_{0}^{\frac{3}{2}}\varepsilon^{\frac{3}{2}}\right)\langle t\rangle^{-\frac{1}{2}+\delta}.
	\end{equation*}
	Based on~\eqref{est:decaypm}, the above estimate and $r\ge 2t$, we have 
	\begin{equation*}
	\sum_{|I|\le N-10}\left|\partial\Gamma^{I}w(t,x)\right|\lesssim \left(\varepsilon+C_{0}^{\frac{3}{2}}\varepsilon^{\frac{3}{2}}\right)\langle t\rangle^{-\frac{1}{2}+\delta}\langle t-r\rangle^{-1}\lesssim \left(\varepsilon+C_{0}^{\frac{3}{2}}\varepsilon^{\frac{3}{2}}\right)\langle t\rangle^{-\frac{1}{2}}\langle t-r\rangle^{-\frac{3}{4}},
	\end{equation*}
	which strictly improves the estimate of $|\partial\Gamma^{I}w|$ in~\eqref{est:Bootw} for $C_{0}$ large enough and $\varepsilon$ small enough (depend on $C_{0}$).
	
	Case II: Let $(t,x)\in \left\{(t,x):r\le 2t\right\}$.
	Similar as in Step 2, we decompose the wave components $w$ as 
	\begin{equation*}
	w=\Upsilon_{0}+\Upsilon_{1}+\partial^{\alpha}\Psi_{\alpha}+\partial_{\alpha}\Phi^{\alpha}.
	\end{equation*}
	We claim that
	\begin{equation}\label{est:pwre2}
	\begin{aligned}
	\sum_{|I|\le N-9}\left(\left|\partial\Gamma^{I}\Upsilon_{0}(t,x)\right|+\left|\partial\Gamma^{I}\Upsilon_{1}(t,x)\right|\right)
	&\lesssim \left(\varepsilon+C_{0}^{2}\varepsilon^{2}\right)\langle t\rangle^{-\frac{1}{2}}\langle t-r\rangle^{-\frac{3}{4}},\\
	\sum_{\substack{|I|\le N-9\\\alpha=0,1,2 }}\left(\left|\partial\partial\Gamma^{I}\Psi_{\alpha}(t,x)\right|+\left|\partial\partial\Gamma^{I}\Phi^{\alpha}(t,x)\right|\right)&\lesssim \left(\varepsilon+C_{0}^{2}\varepsilon^{2}\right)\langle t\rangle^{-\frac{1}{2}}\langle t-r\rangle^{-\frac{3}{4}}.
	\end{aligned}
	\end{equation}
	\emph{Estimate on $\partial \Gamma^{I}\Upsilon_{0}$.} Indeed, from the smallness condition~\eqref{est:small} and~\eqref{equ:wavedecom},
	\begin{equation*}
	\begin{aligned}
	\sum_{|I|\le N-8}\left(\left\|\Gamma^{I}\Upsilon_{0}(0,x)\right\|_{W^{2,1}}+\left\|\partial_{t}\Gamma^{I}\Upsilon_{0}(0,x)\right\|_{W^{1,1}}\right)&\lesssim \varepsilon,\\
	\sum_{|I|\le N-8}\left(\left\|S\Gamma^{I}\Upsilon_{0}(0,x)\right\|_{W^{2,1}}+\left\|\partial_{t}S\Gamma^{I}\Upsilon_{0}(0,x)\right\|_{W^{1,1}}\right)&\lesssim \varepsilon.
	\end{aligned}
	\end{equation*}
	Thus, from $[\Box,S]=2\Box$, $[\Box,\Gamma_{k}]=0$,~\eqref{est:Linwave} and~\eqref{equ:wavedecom} for $\Upsilon_{0}$, we see that
	\begin{equation}\label{est:Up0Linf}
	\sum_{|I|\le N-8}\left(\left\|S\Gamma^{I}\Upsilon_{0}\right\|_{L_{x}^{\infty}}+\left\|\Gamma^{I}\Upsilon_{0}\right\|_{L^{\infty}_{x}}\right)\lesssim \varepsilon \langle t\rangle^{-\frac{1}{2}}.
	\end{equation}
	Based on the above estimates and~\eqref{est:decaypm}, we see that
	\begin{equation}\label{est:Up0}
	\begin{aligned}
	&\langle t-r\rangle\sum_{|I|\le N-9}\left|\partial\Gamma^{I}\Upsilon_{0}(t,x)\right|\\
	&\lesssim \sum_{|I|\le N-8}\left|S\Gamma^{I}\Upsilon_{0}\right\|_{L_{x}^{\infty}}
	+ \sum_{|I|\le N-8}\left\|\Gamma^{I}\Upsilon_{0}\right\|_{L^{\infty}_{x}}\lesssim \varepsilon \langle t\rangle^{-\frac{1}{2}}.
	\end{aligned}
	\end{equation}
	
	\emph{Estimate on $\partial \Gamma^{I}\Upsilon_{1}$.} From~\eqref{est:parGamma},~\eqref{est:Liniwaveinh},~\eqref{est:GBLtx1},~\eqref{est:GSBLtx1},~\eqref{est:GSBLtx2} and~\eqref{equ:wavedecom}, for all $I\in \mathbb{N}^{6}$ with $|I|\le N-8$, we have 
	\begin{equation}\label{est:Up1Linf}
	\begin{aligned}
	\left\|\Gamma^{I}\Upsilon_{1}\right\|_{L_{x}^{\infty}}&\lesssim \langle t\rangle^{-\frac{1}{2}}\int_{0}^{t}(1+s)^{-\frac{1}{2}}\|S\Gamma^{I}G(w,v)\|_{L_{x}^{1}}\d s \\
	&+\langle t\rangle^{-\frac{1}{2}}\sum_{|I|\le N-7}\int_{0}^{t}(1+s)^{-\frac{1}{2}}\|\Gamma^{I}G(w,v)\|_{L_{x}^{1}}\d s
	\lesssim C_{0}^{4}\varepsilon^{4}\langle t\rangle^{-\frac{1}{2}},
	\end{aligned}
	\end{equation}
	and for all $I\in \mathbb{N}^{6}$ with $|I|\le N-9$, we have
	\begin{equation*}
	\begin{aligned}
	\left\|S\Gamma^{I}\Upsilon_{1}\right\|_{L_{x}^{\infty}}&\lesssim \langle t\rangle^{-\frac{1}{2}}\int_{0}^{t}(1+s)^{-\frac{1}{2}}\|S^{2}\Gamma^{I}G(w,v)\|_{L_{x}^{1}}\d s \\
	&+\langle t\rangle^{-\frac{1}{2}}\sum_{|I|\le N-8}\int_{0}^{t}(1+s)^{-\frac{1}{2}}\|S\Gamma^{I}G(w,v)\|_{L_{x}^{1}}\d s\\
	&+\langle t\rangle^{-\frac{1}{2}}\sum_{|I|\le N-8}\int_{0}^{t}\|\Gamma^{I} G(w,v)\|_{L_{x}^{1}}\d s
	\lesssim C_{0}^{4}\varepsilon^{4}\langle t\rangle^{-\frac{1}{2}}.
	\end{aligned}
	\end{equation*}
	From the above two estimates and~\eqref{est:decaypm}, we see that
		\begin{equation}\label{est:Up1}
	\begin{aligned}
	&\langle t-r\rangle\sum_{|I|\le N-9}\left|\partial\Gamma^{I}\Upsilon_{1}(t,x)\right|\\
	&\lesssim \sum_{|I|\le N-9}\left\|S\Gamma^{I}\Upsilon_{1}\right\|_{L_{x}^{\infty}} 
	+ \sum_{|I|\le N-8}\left\|\Gamma^{I}\Upsilon_{1}\right\|_{L^{\infty}_{x}}\lesssim C_{0}^{4} \varepsilon^{4} \langle t\rangle^{-\frac{1}{2}}.
	\end{aligned}
	\end{equation}
	
	\emph{Estimate on $\partial\partial \Gamma^{I}\Psi_{\alpha}$.} To obtain a sharp pointwise decay of $\partial\partial \Gamma^{I}\Psi_{\alpha}$, we introduce a new variable by nonlinear transformation (see also~\cite{TSU}),
	\begin{equation*}
	\widetilde{\Psi}_{\alpha}=\Psi_{\alpha}+C_{1}v\partial_{\alpha} w,\quad \mbox{for}\ \alpha=0,1,2.
	\end{equation*} 
	Using~\eqref{equ:waveKG},~\eqref{equ:wavedecom} and an elementary computation, we have 
	\begin{equation*}
	-\Box \widetilde{\Psi}_{\alpha}=\widetilde{F}_{\alpha}(w,v),\quad 
	\left(\widetilde{\Psi}_{\alpha},\partial_{t}\widetilde{\Psi}_{\alpha}\right)_{|t=0}=C_{1}\left(v\partial_{\alpha}w,\partial_{t}(v\partial_{\alpha} w)\right)_{|t=0},
	\end{equation*}
where
\begin{equation*}
\begin{aligned}
\widetilde{F}_{\alpha}(w,v)
&=\frac{C_{1}^{2}}{2}v^{2}\partial_{\alpha} w+C_{1}v\left(C_{1}\partial_{\alpha}Q_{0}(w,v)+C_{1}^{\gamma\beta}\partial_{\alpha}Q_{\gamma\beta}(w,v)\right)\\
&-2C_{1}Q_{0}(\partial_{\alpha}w,v)+C_{1}\partial_{\alpha}w\left(C_{2}Q_{0}(w,v)+C_{2}^{\alpha\beta}Q_{\alpha\beta}(w,v)\right).
\end{aligned}
\end{equation*}
First, using again the smallness condition~\eqref{est:small}, 
\begin{equation*}
\sum_{|I|\le N-5}\mathcal{E}(0,\Gamma^{I}\widetilde{\Psi}_{\alpha})^{\frac{1}{2}}\lesssim \sum_{|I|\le N-4}\left\|\Gamma^{I}(v\partial w)(0,x)\right\|_{L^{2}_{x}}\lesssim \varepsilon.
\end{equation*}
Then, from~\eqref{est:extraL2} and the definition of $\widetilde{F}_{\alpha}$, we infer
\begin{equation*}
\sum_{|I|\le N-5}\int_{0}^{t}\left\|\Gamma^{I}\widetilde{F}_{\alpha}(w,v)\right\|_{L_{x}^{2}}\d s \lesssim C_{0}^{2}\varepsilon^{2}.
\end{equation*}
	Based on the above two estimates, we have 
	\begin{equation*}
	\sum_{|I|\le N-5}\mathcal{E}(t,\Gamma^{I}\widetilde{\Psi}_{\alpha})^{1/2} \lesssim \sum_{|I|\le N-5}\left(\mathcal{E}(0,\Gamma^{I}\widetilde{\Psi}_{\alpha})^{1/2} +\int_{0}^{t}\left\|\Gamma^{I}\widetilde{F}_{\alpha}(w,v)\right\|_{L_{x}^{2}}\d s\right) \lesssim \varepsilon+C_{0}^{2}\varepsilon^{2}.
	\end{equation*}
	From the above estimates,~\eqref{est:parGamma},~\eqref{est:SobolevGlobal},~\eqref{est:Bootv} and~\eqref{est:pointwavew},
	\begin{equation}\label{est:pxPsi}
	\begin{aligned}
	\sum_{|I|\le N-8}\left|\partial \Gamma^{I}\Psi_{\alpha}\right|
	&\lesssim \sum_{|I|\le N-8}\left(\left|\partial \Gamma^{I}\widetilde{\Psi}_{\alpha}\right|+\left|\Gamma^{I}(v\partial_{\alpha}w)\right|\right)\\
	&\lesssim \left(\varepsilon+C_{0}^{2}\varepsilon^{2}\right)\langle t\rangle^{-\frac{1}{2}}.
	\end{aligned}
	\end{equation}
	Combining the above estimate with~\eqref{est:extwaveHe},~\eqref{est:Bootv} and~\eqref{est:pointwavew}, we obtain 
	\begin{equation}\label{est:Psi}
	\begin{aligned}
	\sum_{|I|\le N-9}\langle t-r\rangle\left|\partial\partial\Gamma^{I}\Psi_{\alpha}(t,x)\right|
	&\lesssim \sum_{|I|\le N-8}\left(\left\|\partial \Gamma^{I}\Psi_{\alpha}\right\|_{L_{x}^{\infty}}+t\left\|\Gamma^{I}(v\partial w)\right\|_{L_{x}^{\infty}}\right)\\
	&\lesssim \left(\varepsilon+C_{0}^{2}\varepsilon^{2}\right)\langle t\rangle^{-\frac{1}{2}}.
	\end{aligned}
	\end{equation}
	
	\emph{Estimate on $\partial \partial \Gamma^{I}\Phi^{\alpha}$.}
	Using a similar argument as above, we also have
		\begin{equation}\label{est:pxPhi}
	\begin{aligned}
	\sum_{|I|\le N-8}\left|\partial \Gamma^{I}\Phi^{\alpha}\right|
	&\lesssim \left(\varepsilon+C_{0}^{2}\varepsilon^{2}\right)\langle t\rangle^{-\frac{1}{2}},
	\end{aligned}
	\end{equation}
	and so using again~\eqref{est:extwaveHe},~\eqref{est:Bootv} and~\eqref{est:pointwavew}, we obtain 
	\begin{equation}\label{est:Phi}
	\sum_{|I|\le N-9}\langle t-r\rangle\left|\partial\partial\Gamma^{I}\Phi^{\alpha}(t,x)\right|
	\lesssim \left(\varepsilon+C_{0}^{2}\varepsilon^{2}\right)\langle t\rangle^{-\frac{1}{2}}.
	\end{equation}
	Combining~\eqref{est:Up0},~\eqref{est:Up1},~\eqref{est:Psi} and~\eqref{est:Phi}, we obtain~\eqref{est:pwre2} which strictly improves  the estimate of $|\partial\Gamma^{I}w|$ in~\eqref{est:Bootw} for $C_{0}$ large enough and $\varepsilon$ small enough (depending on $C_{0}$).
	
		\smallskip
	{\textbf{Step 5.}} Closing the estimate in $\left\| w\right\|_{L_{x}^{\infty}}$. 
    Combining the estimates~\eqref{est:Up0Linf},~\eqref{est:Up1Linf},~\eqref{est:pxPsi} and~\eqref{est:pxPhi}, we obtain
    \begin{equation*}
    \|w\|_{L_{x}^{\infty}}\lesssim 
    \left\|\Upsilon_{0}\right\|_{L_{x}^{\infty}}+\left\|\Upsilon_{1}\right\|_{L_{x}^{\infty}}
    +\left\|\partial \Psi_{\alpha}\right\|_{L_{x}^{\infty}}+\left\|\partial \Phi^{\alpha}\right\|_{L_{x}^{\infty}}\lesssim (\varepsilon+C_{0}^{2}\varepsilon^{2})\langle t\rangle^{-\frac{1}{2}},
    \end{equation*}
    which strictly improves  the estimate of $\|w\|_{L_{x}^{\infty}}$ in~\eqref{est:Bootw} for $C_{0}$ large enough and $\varepsilon$ small enough (depending on $C_{0}$).
	
	\smallskip
	{\textbf{Step 6.}} Closing the estimates of $\Gamma^{I}v$ and $G_{i}\Gamma^{I}v$.
	 First, from~\eqref{est:small},~\eqref{est:L2Q0} and Corollary~\ref{coro:KGdecay}, we directly have
	\begin{equation*}
	\langle t+r\rangle\sum_{|I|\le N-5}\left|\Gamma^{I}v(t,x)\right|\lesssim \varepsilon+ C_{0}^{2}\varepsilon^{2},
	\end{equation*}
	 Then, from~\eqref{est:Bootw} and~\eqref{est:LtLx1Q02}, we see that 
	\begin{equation*}
	\begin{aligned}
	&\sum_{|I|\le N}\int_{0}^{t}\langle s\rangle^{-\delta}\|\Gamma^{I}Q(w,v)\|_{L_{x}^{2}}\|\partial_{t}\Gamma^{I}v\|_{L_{x}^{2}}\d s \\
	&\lesssim C_{0}\varepsilon \sum_{|I|\le N}\int_{0}^{t}\|\Gamma^{I}Q(w,v)\|_{L_{x}^{2}}\d s \lesssim 
	C_{0}^{\frac{5}{2}}\varepsilon^{\frac{5}{2}}\langle t\rangle^{\delta}.
	\end{aligned}
	\end{equation*}
	Therefore, from~\eqref{est:GhostGu} and \eqref{est:inita}, we obtain
	\begin{equation}
		\sum_{|I|\leq N, \,i=1,2}\int_{0}^{t}\langle s \rangle^{-\delta}\int_{\R^{2}}\left(\frac{|\Gamma^{I}v|^{2}}{\langle r-s\rangle^{\frac{3}{2}}}+\frac{|G_{i}\Gamma^{I}v|^{2}}{\langle r-s\rangle^{\frac{3}{2}}}\right)\d x \d s\lesssim \varepsilon^2+C_{0}^{\frac{5}{2}}\varepsilon^{\frac{5}{2}}\langle t\rangle^{\delta}.
	\end{equation}
	These strictly improve the estimates of $\Gamma^{I}v$ and $G_{i}\Gamma^{I}v$ in the bootstrap assumption~\eqref{est:Bootw} for $C_{0}$ large enough and $\varepsilon$ small enough (depending on $C_{0}$).
	
	\smallskip
	At this point, we have strictly improved all the bootstrap estimates of $(w,v)$ in~\eqref{est:Bootw} and~\eqref{est:Bootv}. In conclusion, for all initial data $(\vec{w}_{0},\vec{v}_{0})$ satisfying~\eqref{est:small}, we show that $T_{*}(\vec{w}_{0},\vec{v}_{0})=\infty$ and thus the proof of Proposition~\ref{prop:main} is complete.
\end{proof}

\end{document}